\documentclass{amsart}
\usepackage[utf8]{inputenc}
\usepackage[a4paper, total={6in, 8in}]{geometry}
\usepackage[T1]{fontenc}

\usepackage{amsmath}
\usepackage{bm}
\usepackage{mathtools}
\usepackage{amssymb}
\usepackage{enumitem}

\usepackage{xcolor}
\usepackage{graphicx}
\usepackage{caption}
\usepackage{wrapfig}
\usepackage{float}
\usepackage[justification=centering]{caption}
\usepackage{subcaption}
\usepackage{geometry}
\usepackage{pgf,tikz,pgfplots}
\pgfplotsset{compat=1.15}
\usetikzlibrary{arrows}
\graphicspath{ {./images/} }
\usepackage{tikz-cd}
\usepackage{hyperref}
\usepackage{mathrsfs}
\usepackage{amsthm}

\usepackage{color}
\definecolor{purple}{rgb}{.5, 0, .635} 
\definecolor{green}{rgb}{0.16, 0.62, 0.38}
\definecolor{red}{rgb}{1,.45,.45}

\usepackage{comment}

\DeclarePairedDelimiter{\set}{\{}{\}}

\newcommand{\R}{\mathbb R}

\newcommand{\Z}{\mathbb Z}

\newcommand{\Hyp}{\mathbb H}

\newcommand{\F}{\mathcal{F}}

\DeclareMathOperator{\Homeo}{Homeo}

\DeclareMathOperator{\fr}{fr}
\DeclareMathOperator{\End}{End}
\DeclareMathOperator{\core}{core}

\theoremstyle{plain}
\newtheorem{thm}{Theorem}[section]
\newtheorem{lem}[thm]{Lemma}
\newtheorem{cor}[thm]{Corollary}

\newtheorem{prop}[thm]{Proposition}

\newtheorem{obs}{Observation}

\theoremstyle{definition}
\newtheorem{defn}[thm]{Definition}

\theoremstyle{remark}
\newtheorem{rem}[thm]{Remark}
\title{Simultaneous Laminations}
\author{Isaac Broudy}

\begin{document}

\maketitle

\begin{abstract}
    Calegari introduced the laminations $\Lambda^\pm_u$ associated to a universal circle.  We study the laminations $\Lambda^\pm_u$ for pseudo-Anosov orbit space universal circles of taut foliations, on atoroidal three-manifolds. We prove that $\Lambda^+_u$ and $\Lambda^-_u$ are completely determined by the stable and unstable prelaminations on the boundary of the orbit space. Then, using a result of Barthelm\'e, Bonatti, and Mann, we prove that for any action $\rho:\pi_1(M)\to \Homeo^+(S^1)$ coming from an orbit space of a pseudo-Anosov flow, there is a finite collection of lamination pairs such that $(\Lambda^+_u,\Lambda^-_u)$ lies in this collection for any minimal orbit space universal circle whose action is conjugate to $\rho$. 
\end{abstract}

\section{Introduction}

For any taut foliation $\F$ on a closed atoroidal three-manifold $M$ there exists a universal circle for $\F$ \cite{Thurston1997}, \cite{Thurston1998}, and \cite{CalegariDunfield2003}. A universal circle for $\F$ consists of a topological circle, $S^1_u$, equipped with an action $\rho:\pi_1(M)\to \Homeo^+(S^1_u)$ and an equivariant family of monotone maps $\left\{\phi_\lambda\colon S^1_u\to \partial_\infty\lambda \middle \vert \lambda\in \tilde{\F} \right\}$, where $\tilde{\F}$ is the lift of $\F$ to the universal cover $\widetilde{M}$ and $\partial_\infty\lambda$ is the ideal boundary of a leaf $\lambda$. By work of Calegari \cite{Calegari2006}, for any minimal universal circle there exists a pair of $\pi_1(M)$-invariant laminations of $S^1_u$, denoted $\Lambda^\pm_u$, which may be viewed as analogues of the stable and unstable foliations of a pseudo-Anosov flow.

For a pseudo-Anosov flow $\psi$ on a closed three-manifold $M$, with transverse singular foliations $\F^s_\psi$ and $\F^u_\psi$ on $M$, the orbit space $\mathcal{O}$ is a topological plane and is equipped with a pair of one-dimensional singular foliations, $\bar{\F}^s_\psi$ and $\bar{\F}_\psi^u$, and a $\pi_1(M)$-action preserving each foliation \cite{Barbot1995}, \cite{Fenley1994}, and \cite{FenleyMosher2001}. By work of \cite{Fenley2012} and \cite{Bonatti2024}, any bifoliated plane admits a canonical compactification by a circle, denoted $\partial\mathcal{O}$. Additionally, the action of $\pi_1(M)$ on $\mathcal{O}$ extends to an action of $\pi_1(M)$ on $\partial\mathcal{O}$ by orientation preserving homeomorphisms of $S^1$. Each non-singular leaf determines a single unordered pair of distinct points in $\partial \mathcal{O}$, while each singular $p$-prong determines $p$ unordered pairs which form a $p$-gon of leaves in $\partial\mathcal{O}$. The collection of all unordered pairs of endpoints from stable and unstable leaves is denoted $\partial \mathcal{O}^s$ and $\partial \mathcal{O}^u$, respectively. 

Recent work of Landry, Minsky, and Taylor \cite{LandryMinskyTaylor2024} shows that for any pseudo-Anosov flow $\psi$ on an atoroidal $M$ and any codimension-one foliation $\F$ almost transverse to $\psi$ (i.e. there exists a dynamic blow up of $\psi$ which is transverse to $\F$), there exists a collection of monotone maps $\phi_\lambda$ making $\partial\mathcal{O}$, along with the induced action of $\pi_1(M)$ on $\partial \mathcal{O}$, a universal circle for $\F$.

The aim of this paper is to understand to what extent $\partial\mathcal{O}^{s,u}$ and $\Lambda^\pm_u$ agree as prelaminations of $\partial\mathcal{O}$, and how this depends on the leaf spaces of $\bar\F_\psi^{s,u}$. When $\F$ is a depth-one foliation and $\psi$ is a transverse pseudo-Anosov flow without perfect fits, Huang \cite{Huang2026} proved that $\Lambda^+_u=\partial\mathcal{O}^s$ and $\Lambda^-_u=\partial\mathcal{O}^u$. We establish the relationship between $\Lambda^\pm_u$ and $\partial\mathcal{O}^{s,u}$ for any taut foliation $\F$ and any almost transverse pseudo-Anosov flow $\psi$. In doing so, we will show that the $\Lambda^\pm_u$ laminations contain a significant amount of the bifoliation on $\mathcal{O}$ and additional data about how leaves of $\tilde{\F}$ accumulate on leaves of $\tilde{\F}^{s}_\psi$ and $\tilde{\F}^{u}_\psi$, the lifts of $\F^s_\psi$ and $\F^u_\psi$ to $\widetilde{M}$. Moreover, we will show that given an action $\pi_1(M)\curvearrowright S^1$ coming from an orbit space of a pseudo-Anosov flow, the $\Lambda^\pm_u$ laminations for any minimal orbit space universal circle with the same action are determined by the action up to finite ambiguity. This constrains the possible foliations which are almost transverse to $\psi$.

The leaf space of $\F^{s,u}_\psi$ is either $\R$, a simply connected non-Hausdorff one-manifold, a real tree, or a non-Hausdorff tree \cite{GabaiOertel1989}, \cite{GabaiKazez1997}, and \cite{Fenley2003}. A branching leaf $l\in \bar\F^{s,u}_\psi$ (similarly, a face of $l$, if $l$ is singular) has \textit{branching on one side} if there exists a leaf branching with $l$ in only one connected component of $\mathcal{O}\setminus l$, and has \textit{branching on both sides} otherwise. Note this terminology is distinct from one-sided or two-sided branching of a co-orientable taut foliation.

In general, the leaf space of $\bar\F^s_\psi$ and $\bar\F^u_\psi$ need not be homeomorphic. For this reason, the main theorems are stated and proved for $\bar\F^s_\psi$ and $\bar\F^u_\psi$ separately. Each of our main theorems concern $\partial\mathcal{O}$ which are minimal universal circles (see Theorem \ref{LMTmin} for conditions making $\partial\mathcal{O}$ minimal). The first of our main results concerns $\psi$ which has either a Hausdorff stable or Hausdorff unstable leaf space.
\begin{thm}\label{HausdorffThm}
    Let $\psi$ be a pseudo-Anosov flow on $M$ such that $\F^s_\psi$ ($\F^u_\psi$, respectively) has a Hausdorff leaf space, possibly with perfect fits, and let $\F$ be a foliation almost transverse to $\psi$. Suppose that $(\partial\mathcal{O},\phi_{\lambda})$ is a minimal universal circle. If there exists $\lambda\in L$ such that $\fr\Omega_\lambda$ contains a stable leaf slice (unstable leaf slice, respectively), then $\Lambda^+_u = \partial\mathcal{O}^s$  ($\Lambda^-_u=\partial \mathcal{O}^u$, respectively).
\end{thm}
See Section \ref{subsec: pAprelim} for the definition of a leaf slice. The collections $\partial \mathcal{O}^s$ and $\partial \mathcal{O}^u$ define a pair of transverse laminations of $\partial \mathcal{O}$ if and only if $\psi$ has Hausdorff leaf spaces \cite{BarthelmeBonattiMann2025}. Therefore, we can not have a direct analogue of Theorem \ref{HausdorffThm} in the case where $\psi$ has non-Hausdorff leaf spaces. The next theorem applies to $\psi$ with either a stable or unstable foliation which is non-Hausdorff, whose branching leaves branch on one side only.
\begin{thm}\label{OneSided}
    Let $\psi$ be a pseudo-Anosov flow on $M$ such that $\F^s_\psi$ ($\F^u_\psi$, respectively) is non-Hausdorff and contains no leaves branching on both sides, and let $\F$ be a foliation almost transverse to $\psi$. Suppose that $(\partial\mathcal{O},\phi_\lambda)$ is a minimal universal circle. If there exists $\lambda\in L$ such that $\fr\Omega_\lambda$ contains a stable leaf slice (unstable leaf slice, respectively), then $\partial\mathcal{O}^s\subsetneq\Lambda^+_u$ ($\partial\mathcal{O}^u\subsetneq\Lambda^-_u$, respectively).
\end{thm}
We also obtain a result regarding the closure of $\partial \mathcal{O}^{s,u}$ in the space of unordered pairs of distinct points, denoted by $\overline{\partial\mathcal{O}^{s,u}}$, which is a lamination of $\partial \mathcal{O}$. The following corollary is a consequence of the proof of Theorem \ref{OneSided}.
\begin{cor}\label{OneSidedCor}
    Let $\psi$ be a pseudo-Anosov flow on $M$ such that $\F^s_\psi$ ($\F^u_\psi$, respectively) contains no leaves branching on both sides, and let $\F$ be a foliation almost transverse to $\psi$. Suppose that $(\partial\mathcal{O},\phi_\lambda)$ is a minimal universal circle. If there exists $\lambda\in L$ such that $\fr\Omega_\lambda$ contains a stable leaf slice (unstable leaf slice, respectively), then $\overline{\partial\mathcal{O}^s}\subseteq\Lambda^+_u$  ($\overline{\partial\mathcal{O}^u}\subseteq\Lambda^-_u$, respectively).
\end{cor}
In general, one can not promote the conclusion of Corollary \ref{OneSidedCor} to an equality due to the presence of \textit{diagonal leaves} (see Section \ref{sec:Reconstruction}). The result of Corollary \ref{OneSidedCor} is sharp. In Section \ref{sec:Example}, we construct a pseudo-Anosov flow which is transverse to a finite depth foliation such that $\overline{\partial\mathcal{O}^s}\subsetneq \Lambda^+_u$. The existence of such examples demonstrates that $\Lambda^\pm_u$ are in general not minimal laminations (i.e. there exists a proper sublamination in $\Lambda^\pm_u$).  

Finally, the next theorem addresses the case where $\psi$ has a stable or unstable foliation which contains leaves with branching on both sides.
\begin{thm}\label{TwoSided}
    Let $\psi$ be a pseudo-Anosov flow on $M$ such that $\F^s_\psi$ ($\F^u_\psi$, respectively) is non-Hausdorff and contains a leaf branching on both sides, and let $\F$ be a foliation almost transverse to $\psi$. Suppose that $(\partial\mathcal{O},\phi_\lambda)$ is a minimal universal circle. If there exists $\lambda\in L$ such that $\fr\Omega_\lambda$ contains a stable leaf slice (unstable leaf slice, respectively), then all but finitely many leaves of $\partial\mathcal{O}^{s}$ ($\partial\mathcal{O}^{u}$, respectively) are contained in $\Lambda^+_u$ ($\Lambda^-_u$, respectively), up to the action of $\pi_1(M)$.
\end{thm}
With these results at hand, we prove that one can recover all possible laminations for $\Lambda^\pm_u$ from the action of $\pi_1(M)$ on $\partial\mathcal{O}$ alone.
\begin{thm}\label{reconstruct}
    Let $\F$ be a taut foliation on $M$ and $\rho:\pi_1(M)\to\Homeo^+(S^1)$ be an action coming from the orbit space of a pseudo-Anosov flow. Then there exists a finite set $S$ of lamination pairs such that for any minimal orbit space universal circle $S^1_u$ for $\F$ with an action conjugate to $\rho$, the associated $(\Lambda^+_u,\Lambda^-_u)$ lies in $S$.
\end{thm}

\subsection*{Outline} The organization of this paper is as follows. In Section \ref{sec:Preliminaries}, we recall the definition of universal circle and the construction of the laminations $\Lambda^\pm_u$. We then summarize basic results on pseudo-Anosov flows and their orbit spaces, dynamic blow-ups of pseudo-Anosov flows and the shadow of transverse leaves, and recent results of Landry, Minsky, and Taylor \cite{LandryMinskyTaylor2024} on orbit space universal circles. In Section \ref{sec:Main}, we prove Theorems \ref{HausdorffThm}, \ref{OneSided}, and \ref{TwoSided}, as well as Corollary \ref{OneSidedCor}, and we prove Theorem \ref{reconstruct} in Section \ref{sec:Reconstruction}. Then, in Section \ref{sec:Example}, we construct an example of a pseudo-Anosov flow transverse to a finite depth foliation with the property that $\overline{\partial\mathcal{O}^s}\subsetneq \Lambda^+_u$. 

\subsection*{Acknowledgments} The author would like to express their gratitude to Kathryn Mann for her guidance and support throughout the course of this research. The author is also grateful to Michael Landry, S\'ergio Fenley, Sam Taylor, and Chi Cheuk Tsang for insightful discussions and helpful suggestions. Finally, the author would like to thank Junzhi Huang and Chaitanya Tappu for comments on a preliminary version of this paper.

\section{Preliminaries} \label{sec:Preliminaries}

\subsection{Taut Foliations, Universal Circles, and Laminations}
In this section we introduce the necessary machinery and results from the theory of taut foliations needed to give the definition of a universal circle and outline the construction of the laminations $\Lambda^\pm_u$. 

Let $\F$ be a \textit{taut foliation} of a closed, oriented, irreducible three-manifold $M$. That is, $\F$ is a codimension-one foliation of $M$ and there exists a closed loop intersecting every leaf, transverse to $\F$.  A well-known theorem of Novikov \cite{Novikov1965} implies that every leaf of $\F$ is incompressible and every closed transversal to $\F$ is essential in $\pi_1(M)$. Hence, the foliation lifted to the universal cover $\widetilde{M}$, denoted by $\tilde{\F}$, is a foliation by planes. 

The space obtained by collapsing each leaf of $\tilde{\F}$ (i.e. the space $\widetilde{M}/\sim_{\tilde{\F}}$ where $p\sim_{\tilde{\F}} q$ if and only if $p,q\in \lambda$, for some $\lambda\in\tilde{\F}$) is called the \textit{leaf space} of $\F$, denoted $L$. It follows from Novikov's theorem that $L$ is a connected, simply connected one-manifold, which is possibly non-Hausdorff. $\F$ is said to be \textit{$\R$-covered} when $L$ is Hausdorff. The leaves corresponding to Hausdorff points of $L$ are called \textit{separated leaves} and the non-Hausdorff points of $L$ are called \textit{nonseparated} or \textit{branching leaves}. A \textit{cataclysm} in $L$ is a maximal set of branching leaves such that all leaves in this set are nonseparated with each other.

In this paper, we consider taut foliations $\F$ which are \textit{co-oriented}, meaning there is a consistent choice of transverse orientation for each leaf. If $\F$ is co-oriented, $L$ inherits an orientation and a partial order $<$, where $\lambda<\mu$ if there exists a positively oriented interval in $L$ from $\lambda$ to $\mu$. Two leaves $\lambda,\lambda'\in L$ for which there does not exist an oriented embedded interval from $\lambda$ to $\lambda'$ and vice versa are said to be \textit{incomparable}. A co-oriented taut foliation $\F$ is said to have \textit{positive branching} (\textit{negative branching}, respectively), if there exists a pair of branching leaves $\lambda_i$ and a leaf $\mu$ such that $\mu < \lambda_i$, for $i=1,2$ ($\lambda_i<\mu$ for $i=1,2$, respectively). 

If we further assume that $M$ is atoroidal, then a theorem of Candel \cite{Candel1993} implies that there exists a \textit{leafwise hyperbolic} Riemannian metric $g$ on $M$---i.e. the pullback of $g$ on each leaf of $\F$ is a hyperbolic metric. Hence each leaf $\lambda \in \tilde{\F}$ is a hyperbolic plane and has an ideal boundary, $\partial_\infty\lambda$, homeomorphic to a circle. The union of the ideal boundaries of each leaf, $E_\infty \coloneqq  \bigsqcup_{\lambda\in L} \partial_\infty \lambda$, forms the total space of a circle bundle over $L$. We refer the reader to \cite{CalegariDunfield2003} for a discussion of the topology on $E_\infty$. Since $\pi_1(M)$ acts on $\widetilde{M}$ preserving $\widetilde{\F}$, there is an induced action of $\pi_1(M)$ on $E_\infty$ preserving fibers. The bundle $E_\infty$ will allow us to understand the global geometry of leaves in $\F$ and define universal circles. 

A universal circle is an object which collates the ideal boundaries of each leaf $\tilde{\F}$, thereby encoding the structure of $L$ and the action of $\pi_1(M)$ on $\widetilde{M}$. Fundamental to the definition of a universal circle are monotone maps of circles.  
\begin{defn}
   Suppose that $f:X\to Y$ is a continuous map between two oriented circles $X$ and $Y$. The map $f$ is a \textit{monotone map} of circles if $f$ has degree 1 and the preimage of every $y\in Y$ is contractible.
\end{defn}
A maximal connected open interval in $X$ whose image is a point is called a \textit{gap of $f$} and the complement of the union of all the gaps of $f$ is called the \textit{core of $f$}, denoted $\core(f)$. 

We now state the definition of a universal circle given in \cite{Calegari2006}.
\begin{defn}\label{UC}
    A \textit{universal circle} for $\F$ is an oriented topological circle $S^1_u$ equipped with a family of monotone maps $\left\{ \phi_\lambda\colon S^1_u\to \partial_\infty\lambda \middle\vert \lambda\in L \right\}$, together with an action $\rho:\pi_1(M)\to \Homeo^+(S_u^1)$ such that the following hold:
    \begin{enumerate}
        \item The map $\phi:L\times S^1_u\to E_\infty$ given by $\phi(\lambda,p):= \phi_\lambda(p)$ is continuous.
        \item For each $\lambda\in L$ and each $g\in \pi_1(M)$, the following diagram commutes:
        \[\begin{tikzcd}
    	{S^1_u} & {S^1_u} \\
    	{\partial_\infty\lambda} & {\partial_\infty (g\cdot \lambda)}
    	\arrow["{\rho(g)}", from=1-1, to=1-2]
    	\arrow["{\phi_\lambda}"', from=1-1, to=2-1]
    	\arrow["{\phi_{g\cdot\lambda}}", from=1-2, to=2-2]
    	\arrow["g"', from=2-1, to=2-2]
        \end{tikzcd}\]
    \item For any two incomparable leaves $\lambda, \lambda'\in L$, $\core(\phi_\lambda)$ is contained in a single gap of $\phi_{\lambda'}$, and vice versa.
    \end{enumerate}
\end{defn}
Calegari--Dunfield \cite{CalegariDunfield2003} proved that every taut foliation on a closed, oriented, atoroidal three-manifold has a universal circle. A taut foliation may have many distinct universal circles up to conjugacy (see \cite{BuckminsterTaylor2025} for an example of a family of foliations with distinct universal circles). In this paper we restrict our attention to universal circles which are minimal.
\begin{defn}
    A universal circle is said to be \textit{minimal} given that for any $p,q \in S^1_u$ distinct there exists a $\lambda\in L$ such that $\phi_\lambda(p)\neq \phi_\lambda(q)$. 
\end{defn}

Calegari constructed two $\pi_1(M)$-invariant laminations of a minimal universal circle, $\Lambda^+_u$ and $\Lambda^-_u$, which encode the data of the monotone maps. 
\begin{defn}
    A \textit{lamination} of $S^1$ is a closed subset of the space of unordered pairs of distinct points in $S^1$ such that any two elements are \textit{unlinked}---that is, for any two elements $\set{a_1,a_2}, \set{b_1,b_2}$, $b_1$ and $b_2$ lie in the same component of $S^1\setminus\set{a_1,a_2}$. 
\end{defn}
In the subsequent sections, we will work with a generalization of a lamination, called a \textit{prelamination}, which is a set of unlinked pairs of distinct points (i.e. the set is not necessarily closed). Any prelamination can be promoted to a lamination by taking the closure of the prelamination in the space of unordered pairs of distinct points in $S^1$.

We now outline Calegari's construction of $\Lambda^\pm_u$. For any $\lambda \in L$, let $L^\pm(\lambda)$ denote the connected component of $L\setminus\set{\lambda}$ which consist of all leaves on the $\pm$-side of $\lambda$. For any subset $Z\subseteq L$, define $$\core(Z)\coloneqq \overline{\bigcup_{\lambda\in Z}\core(\phi_\lambda)} \subseteq S^1_u.$$ A lamination of $S^1_u$ is obtained by identifying $S^1_u$ with $\partial \Hyp^2$. Using this identification, $\core(Z)$ can be given a convex hull in $\Hyp^2$, denoted $CH(\core(Z))\subseteq\Hyp^2$. The endpoints of the boundary geodesics of $CH(\core(Z))$ give a collection of unordered pairs of distinct points in $S^1_u$, denoted by $\Lambda(Z)$. Doing so using $\core(L^\pm(\lambda))$ yields a pair of laminations associated to each $\lambda$: $$\Lambda^\pm(\lambda) : = \Lambda(\core(L^\pm(\lambda))).$$ Doing this for each leaf yields the laminations $\Lambda^\pm_u$ of $S^1_u$: $$\Lambda^\pm_u := \overline{\bigcup_{\lambda\in L}\Lambda^\pm(\lambda)}.$$
Calegari \cite{Calegari2006} proved that $\Lambda^\pm_u$ form a pair of $\pi_1(M)$-invariant laminations of $S^1_u$.
\begin{thm}[{\cite[Theorem 5.2.3]{Calegari2006}}]\label{CalegariLamination}
Let $M$ be an atoroidal three-manifold with a taut foliation $\F$. If $S^1_u$ is a minimal universal circle for $\F$, then $\Lambda^\pm_u$ are $\pi_1(M)$-invariant laminations of $S^1_u$. Furthermore, if $\F$ has positive branching (negative branching, respectively), then $\Lambda^+_u$ is nonempty ($\Lambda^-_u$, respectively).
\end{thm}
\subsection{Pseudo-Anosov Flows and their Orbit Spaces}\label{subsec: pAprelim}
We now outline some pre-requisites from the theory of pseudo-Anosov flows. In this section, we introduce some of the basic properties of pseudo-Anosov flows and their orbit spaces. For a rigorous treatment of the theory of pseudo-Anosov flows we refer the reader to \cite{BarthelmeMann2026} for a detailed exposition.

Roughly speaking, a pseudo-Anosov flow $\psi$ is a flow on a closed three-manifold which is Anosov away from a finite collection of periodic orbits, called singular orbits, and each neighborhood of a singular orbit is modeled on a semi-branched cover of degree $\frac{p}{2}$ of a periodic orbit of an Anosov flow, for $p\geq 3$. See Figure \ref{fig:dynamicblowup} for a local picture of a singular orbit. As such, $\psi$ comes equipped with a pair of transverse singular foliations, $\F^s_\psi$ and $\F^u_\psi$, called the \textit{stable} and \textit{unstable} foliations, respectively. Each leaf of $\F^s_\psi$ and $\F^u_\psi$ is saturated by orbits, and orbits in the same stable leaf are asymptotic forwards in time, while orbits in the same unstable leaf are asymptotic backwards in time.

If $\psi$ is Anosov, then the foliations $\F^s_\psi$ and $\F^u_\psi$ are taut foliations. Hence,  $\F^s_\psi$ and $\F^u_\psi$ have leaf spaces which are simply connected one-manifolds. If $\psi$ has singular orbits, the leaf spaces of $\F^s_\psi$ and $\F^u_\psi$ are no longer one-manifolds, but are real trees or possibly non-Hausdorff trees (see \cite{GabaiOertel1989},\cite{GabaiKazez1997}, and \cite{Fenley2003}). 

Given a pseudo-Anosov flow, a \textit{dynamic blow-up} is an operation performed on a collection of singular orbits, where each singular orbit in the collection is replaced with an annulus. The annuli introduced in this procedure are called \textit{blow-up annuli}. Each blow-up annulus is invariant under the resulting flow and the two boundary components of each annulus form closed orbits, one of which is attracting and the other repelling. See Figure \ref{fig:dynamicblowup}. Performing dynamic blow-ups on a collection of singular orbits results in an \textit{almost pseudo-Anosov} flow, $\varphi$. We denote the stable and unstable singular foliations of $\varphi$ by $\F^s$ and $\F^u$, where blow-up annuli are contained in both $\F^s$ and $\F^u$. For a precise treatment of dynamic blow-ups see {\cite[Section 3]{LandryMinskyTaylor2026}}. In Section \ref{a-transverse}, dynamic blow-ups will be used to make $\psi$ transverse to $\F$.

Let $\tilde{\varphi},\tilde{\F}^s,\tilde{\F}^u$ denote the lifts of $\varphi$, $\F^s$, and $\F^u$ to the universal cover of $M$, $\widetilde{M}$. The space obtained by collapsing each orbit of $\tilde{\varphi}$ in $\widetilde{M}$ is called the orbit space of $\varphi$, denoted $\mathcal{O}$ and $\pi:\widetilde{M}\to \mathcal{O}$ is the projection map. Fenley--Mosher \cite{FenleyMosher2001} proved that $\mathcal{O}$ is a topological plane and the foliations $\tilde{\F}^s$ and $\tilde{\F}^u$ descend to the orbit space giving a pair of 1-dimensional singular foliations, which we denote by $\bar{\F}^s$ and $\bar{\F}^u$. The image of blow-up annuli are called \textit{blow-up segments}. The foliations $\bar{\F}^s$ and $\bar{\F}^u$ are transverse away from the blow-up segments. Moreover, the action of $\pi_1(M)$ on $\widetilde{M}$ descends to an action on $\mathcal{O}$ by homeomorphisms and preserves $\bar\F^s$ and $\bar\F^u$. 

In \cite{Fenley2012} and \cite{Bonatti2024} it was shown that there is a canonical compactification of $\mathcal{O}$ by a circle, $\partial \mathcal{O}$, into a closed disk, $\bar{\mathcal{O}}$, where each foliation ray limits to a single point in $\partial \mathcal{O}$. Moreover, the action of $\pi_1(M)$ on $\mathcal{O}$ induces an action of $\pi_1(M)$ on $\partial\mathcal{O}$ by homeomorphisms. 

A properly embedded $\R$ in a leaf $l \in \bar\F^{s}$ is called a \textit{leaf slice}. Thus, a nonsingular leaf contains only one leaf slice, while singular leaves contain distinct leaf slices. Furthermore, a leaf slice $f$ is said to be a \textit{face} of $l\in \bar\F^s$ (some authors call such a leaf slice \textit{regular}) if there is a connected component of $\mathcal{O}\setminus f$ containing exactly one unstable half leaf emanating out of $l$. Each leaf slice $l$ determines an unordered pair of points $\End(l)$ consisting of both endpoints of $l$ in $\partial\mathcal{O}$, which are necessarily distinct. Given a finite sequence of distinct leaf slices $l_1,\ldots,l_n$ such that consecutive leaf slices share an endpoint, $\End(\bigcup_{i=1}^n l_i)$ is defined to be the endpoints of $l_1$ and $l_n$ which are not in common with the other leaf slices in the sequence. Hence, $\bar\F^{s}$ defines a collection of unordered pairs of points in $\partial\mathcal{O}$: $$\displaystyle\partial\mathcal{O}^{s} := \bigcup_{\substack{f \subset l\in \bar\F^{s} \\ f \text{ a face}}} \End(f).$$ Leaf slices, faces of leaves, and $\partial\mathcal{O}^u$ are defined analogously using leaves of $\bar\F^u$. In general, $\partial \mathcal{O}^{s,u}$ are only a pair of prelaminations of $\partial\mathcal{O}$. Since $\bar\F^{s,u}$ are preserved by the action of $\pi_1(M)$, the prelaminations $\partial\mathcal{O}^{s,u}$ are $\pi_1(M)$-invariant. Finally, we note that by construction, the prelaminations $\partial\mathcal{O}^{s,u}$ do not depend on the choice of dynamic blow-up of $\psi$.

\subsection{Almost Transverse Flows and Orbit Space Universal Circles}\label{a-transverse}

The aim of this section is to discuss recent work of \cite{LandryMinskyTaylor2024} which shows that for a pseudo-Anosov flow $\psi$ which is almost transverse to a foliation $\F$, there exists a family of monotone maps making $\partial\mathcal{O}$ a universal circle---thereby, combining the machinery of the previous sections. Before doing so, we will first introduce almost transversality and then discuss result of \cite{Fen09} which characterizes the projection of leaves in $\F$ to $\mathcal{O}$.

Consider a foliation $\F$ and pseudo-Anosov flow $\psi$ on a closed, orientable, atoroidal three-manifold $M$. If there exists a dynamic blow-up of $\psi$, denoted $\varphi$, such that $\varphi$ is transverse to $\F$, then $\F$ and $\psi$ are said to be \textit{almost transverse}. We will take $\varphi$ to be the flow obtained by blowing up a minimal collection of singular orbits, with respect to inclusion.  

\begin{figure}[H]
    \centering
    \includegraphics[width= 3.5in]{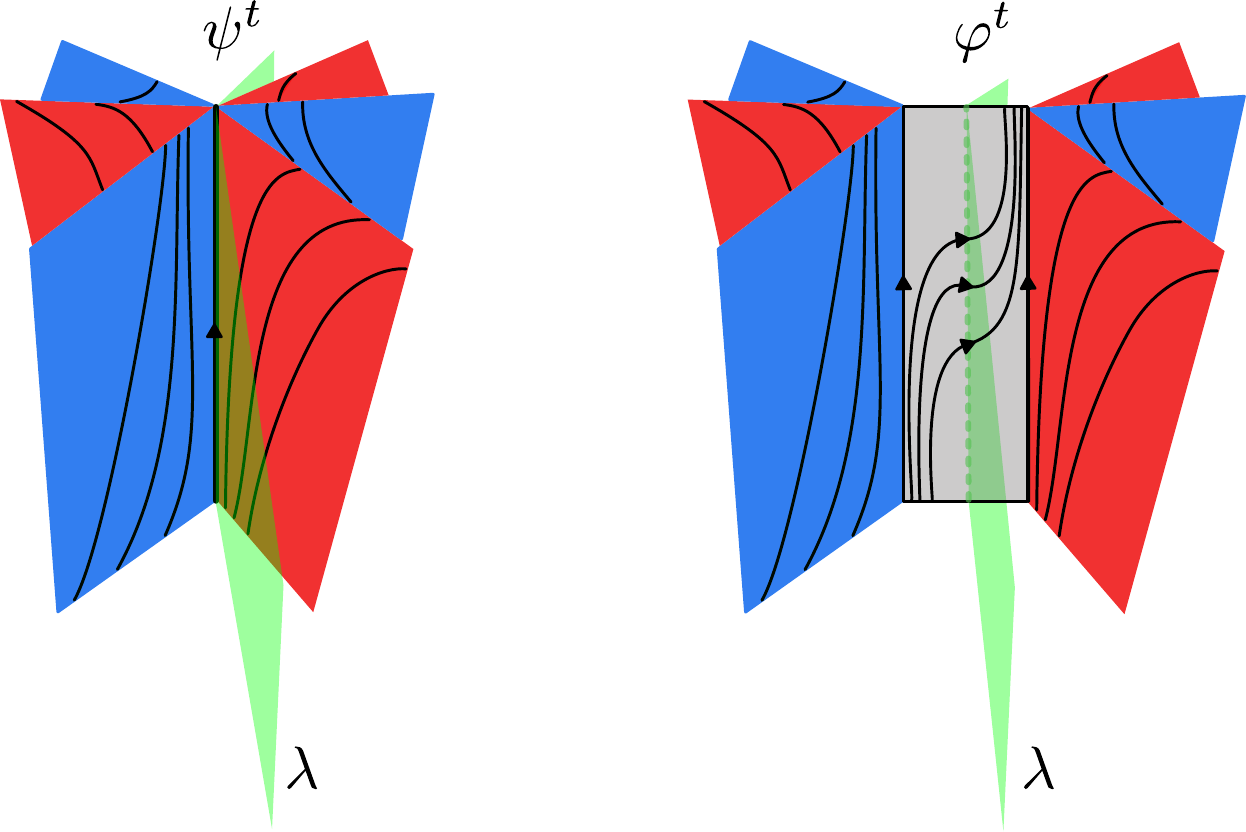}
    \caption{Left: The pseudo-Anosov flow $\psi$ fails to be transverse to $\lambda\in \F$ at a singular orbit. Right: After a dynamic blow-up of the singular orbit, the flow $\varphi$ is transverse to $\lambda\in\F$.}
    \label{fig:dynamicblowup}
\end{figure}

By work of Mosher {\cite[Proposition 2.7]{Mosher92}}, any pseudo-Anosov flow $\psi$ on an atoroidal $M$ is transitive (i.e. contains a dense orbit). Applying the pseudo-Anosov closing lemma {\cite[Proposition 1.4.4]{BarthelmeMann2026}} to the dense orbit of $\varphi$, gives that any $\F$ almost transverse to $\psi$ is taut. Moreover, the transversality of $\varphi$ to $\F$ implies $\F$ is co-orientable.

Given a foliation which is almost transverse to $\psi$, the projection of a leaf $\lambda\in\tilde{\F}$ to $\mathcal{O}$, which we denote by $\Omega_\lambda:= \pi(\lambda)$, is called the \textit{shadow} of $\lambda$. A leaf $\lambda\in\tilde{\F}$ accumulates on a face $\tilde{f}\subset\tilde{l}\in \tilde{\F}^s$ \textit{going up with the flow} (\textit{going down with the flow}, respectively) if for any small disk $D$ transverse to $\tilde\varphi$ and based at a point $x\in \tilde{f}$, and any sequence $\set{x_n}\subset D$ on the side of $\tilde{f}$ containing $\lambda$ such that $x_n\to x$, there exists a sequence $\set{t_n}\subset \R$ such that $\tilde\varphi^{t_n}(x_n)$ lies in $\lambda$ and $t_n\to \infty$ ($t_n\to -\infty$, respectively). See Figure \ref{fig:accumulation1}.

\begin{figure}[H]
    \centering
    \includegraphics[width= 4in]{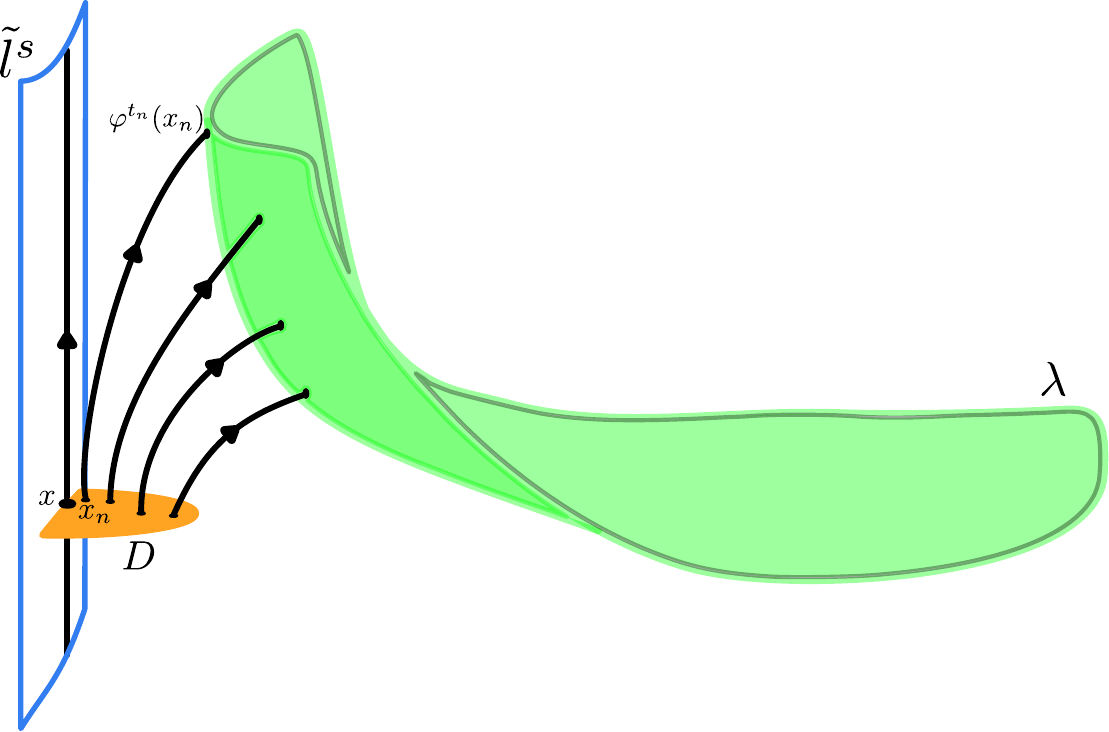}
    \caption{The leaf $\lambda$ accumulates on  $\tilde{l}^s\in\tilde{\F^s}$ going up with the flow. As the sequence $\set{x_n}$ converges to $x$, the time, $t_n$, it takes to flow along $\tilde\varphi$ from $x_n$ to $\lambda$ grows without bound.}
    \label{fig:accumulation1}
\end{figure}

Fenley proved the following important property of the shadows of leaves:
\begin{prop}[{\cite[Proposition 4.1]{Fen09}}]\label{fenleyprop}
    Let $\lambda\in\tilde{\F}$. Then $\Omega_\lambda$ is an open disk whose boundary in $\mathcal{O}$ consists of faces of leaves in $\bar{\F}^s$ and $\bar{\F}^u$. Moreover, if $\Omega_\lambda$ contains a face $f$ of $l\in \bar\F^s$ ($l\in \bar\F^u$, respectively) in its boundary, then $\lambda$ accumulates on $\tilde{f}$ going up with the flow (down with the flow, respectively), where $\tilde f$ is the lift of $f$ to $\widetilde{M}$.
\end{prop}

It follows from Proposition \ref{fenleyprop} that the frontier of $\Omega_\lambda$ in $\mathcal{O}$ is a union of faces of stable and unstable leaves, denoted by $\fr\Omega_\lambda$. See Figure \ref{fig:shadowfig}.

Recently, Landry, Minsky, and Taylor \cite{LandryMinskyTaylor2024} showed the following:
\begin{thm}[{\cite[Theorem 1.1]{LandryMinskyTaylor2024}}]\label{LMTthm1}
    Let $M$ be a closed, oriented, atoroidal three-manifold with a pseudo-Anosov flow $\psi$ and an almost transverse foliation $\F$. Then there exists a family of monotone maps $\set{\phi_\lambda}_{\lambda\in L}$ together with the induced action of $\pi_1(M)$ on $\partial\mathcal{O}$ making $\partial\mathcal{O}$ a universal circle for $\F$.
\end{thm}
They further provide a characterization of the gaps of the monotone maps $\set{\phi_\lambda}_{\lambda\in L}$ which we make use of repeatedly.
\begin{thm}[{\cite[Theorem 1.2]{LandryMinskyTaylor2024}}]\label{LMTthm2}
    For each $\lambda\in L$, each gap of $\phi_\lambda$ is spanned by a finite sequence of frontier components of $\Omega_\lambda$, each sharing an ideal endpoint with the next.
\end{thm}
The union of leaf slices in $\fr\Omega_\lambda$ bounding a single gap is called a \textit{frontier chain}. 

To construct the laminations $\Lambda^\pm_u$ for an orbit space universal circle, the universal circle must be minimal, by Theorem \ref{CalegariLamination}. Landry, Minsky, and Taylor give a necessary and sufficient set of conditions for an orbit space universal circle to be minimal.
\begin{thm}[{\cite[Proposition 7.2]{LandryMinskyTaylor2024}}]\label{LMTmin}
    An orbit space universal circle is minimal if and only if at least one of the following holds:
    \begin{itemize}
        \item $\psi$ is not skew Anosov,
        \item $\F$ is not $\R$-covered,
        \item or $\psi$ is regulating for $\F$.
    \end{itemize}
\end{thm}
A flow $\psi$ is said to be \textit{regulating} for $\F$ given that every flow line of $\tilde{\psi}$ intersects every leaf of $\tilde{\F}$. A pseudo-Anosov flow $\psi$ is \textit{skew Anosov} if its orbit space is isomorphic to the skew bifoliated plane (see {\cite[Definition 2.3.5]{BarthelmeMann2026}}).

\begin{figure}[H]
    \centering
    \includegraphics[width= 4in]{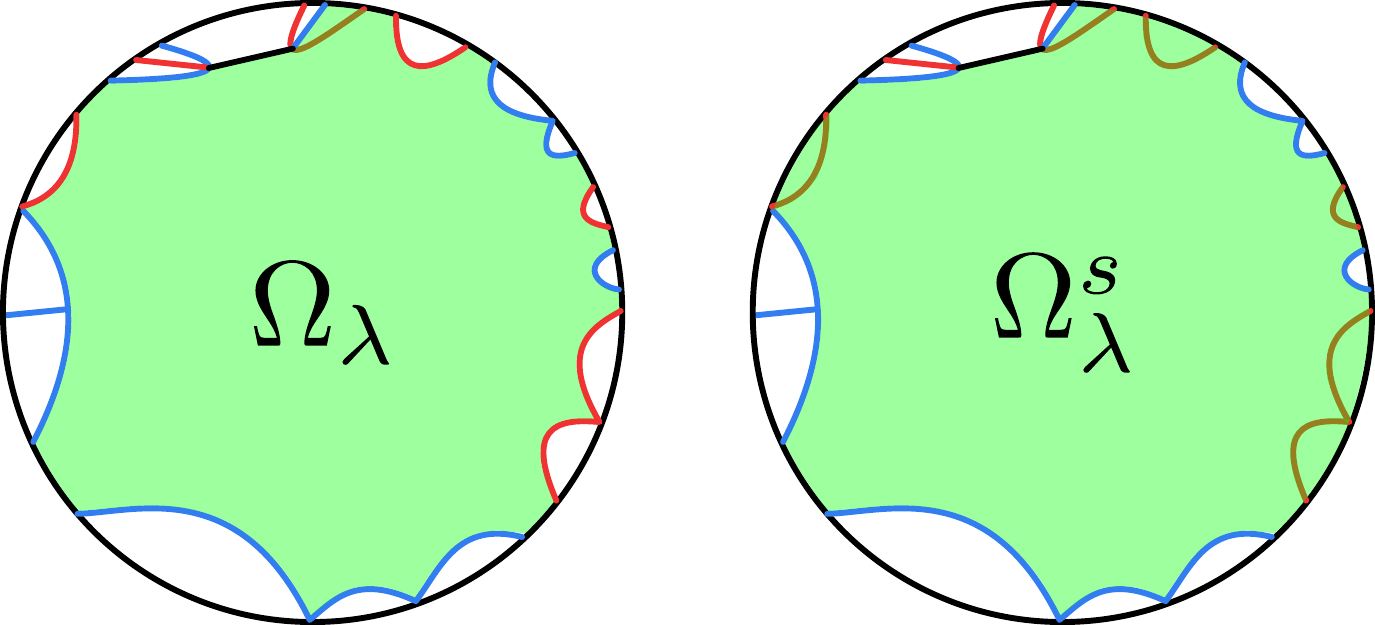}
    \caption{Left: The shadow, $\Omega_\lambda$, of a leaf $\lambda\in L$. Each gap of $\phi_\lambda$ is spanned by a finite chain of leaves in $\fr\Omega_\lambda$, the frontier of $\Omega_\lambda$. Right: The set $\Omega_\lambda^s$ associated to the shadow of $\Omega_\lambda$. Observe that $\Omega_\lambda\subseteq \Omega_\lambda^s$. }
    \label{fig:shadowfig}
\end{figure}

\section{Main Results}\label{sec:Main}

Throughout this section let $M$ be an oriented, closed, irreducible, atoroidal three-manifold, with a pseudo-Anosov flow $\psi$ almost transverse to a foliation $\F$. Let $\varphi$ be a minimal dynamic blow up of $\psi$ transverse to $\F$. We co-orient $\F$ using $\varphi$ and let $L$ be the leaf space of $\F$. The goal of this section is to prove Theorems \ref{HausdorffThm}, \ref{OneSided}, and \ref{TwoSided}.

\subsection{Dynamics of $\pi_1(M)\curvearrowright \mathcal{O}$}

Throughout the proof of the main theorems we make extensive use of the dynamics of $\pi_1(M)\curvearrowright\mathcal{O}$. We now review some relevant properties of this action. For a detailed exposition on the dynamics of $\pi_1(M)\curvearrowright\mathcal{O}$ see \cite{BarthelmeMann2026}. The next theorem gives four important properties we use repeatedly.
\begin{thm}[\cite{BarthelmeFrankelMann2025}, \cite{Fen98}]\label{Anosov-like}
The action of $\pi_1(M)$ on $\mathcal{O}$ induced by a pseudo-Anosov flow $\psi$ satisfies the following properties:
\begin{enumerate}
    \item Let $g\in \pi_1(M)$ and $g\neq id$. If $g$ fixes a leaf of $l\in \bar\F^s_\psi \cup \bar\F^u_\psi$, then $g$ has a unique fixed point in $x\in l$ and, up to taking inverses, is topologically expanding on the $\bar\F^s_\psi$ leaf through $x$ and is topologically contracting on the $\bar\F^u_\psi$ leaf through $x$.
    \item The union of $\bar\F^s_\psi$ leaves ($\bar\F^u_\psi$ leaves, respectively) with nontrivial stabilizer is dense in $\mathcal{O}$.
    \item Every singular leaf in $\bar\F^s_\psi\cup\bar\F^u_\psi$ has a nontrivial stabilizer.
    \item Every branching leaf in $\bar\F^s_\psi\cup\bar\F^u_\psi$ has a nontrivial stabilizer.
\end{enumerate}
\end{thm}
If $\varphi$ is a dynamic blow up of a pseudo-Anosov flow $\psi$, the orbit space of $\varphi$ is a plane $\mathcal{O}_\varphi\cong \R^2$ \cite{FenleyMosher2001} and the action $\pi_1(M) \curvearrowright \mathcal{O}_\varphi$ is semi-conjugate to the action $\pi_1(M)\curvearrowright \mathcal{O}_\psi$ via a map collapsing blow-up segments  to singular orbits of $\psi$. Another important property of the $\pi_1(M)$ action for transitive pseudo-Anosov flows, which we make use of in the proof of the main theorems, is that every leaf $l\in \bar\F^s\cup \bar\F^u$ has a dense orbit, $\pi_1(M)\cdot l$, in $\mathcal{O}$.

In the proof of Theorem \ref{TwoSided} and Theorem \ref{reconstruct} we will make use of the following result of Fenley which controls the number of branching leaves in $\bar\F^s_\psi$ and $\bar{\F}^u_\psi$:
\begin{thm}[\cite{Fen98}]\label{Fenleybranching}
    If $M$ is atoroidal then there are finitely many $\pi_1(M)$-orbits of branching leaves in $\bar\F^s_\psi$ and $\bar\F^u_\psi$. In particular, the number of leaves fixed by an arbitrary $g\in \pi_1(M)$ is finite and uniformly bounded, independent of $g$.
\end{thm}

A point $p\in \mathcal{O}$ is said to be a \textit{noncorner} point if there exists a $g$ fixing $p$ and no other points in $\mathcal{O}$. We call a leaf $l\in \bar\F^{s,u}$ \textit{regular} if it is not a non-separated leaf and is neither a singular nor a blown up leaf. For any $g\in \pi_1(M)$ fixing a regular leaf containing a noncorner point, it follows from Theorem \ref{Anosov-like} that $g$ acts on $\partial\mathcal{O}$ with source-sink dynamics, with two sources and two sinks. We highlight this class of leaves and periodic orbits as they feature throughout the proof of our main theorems.

\subsection{Shadows and Monotone Maps}
Before proceeding to the proof of the main results, we first characterize how the shadows of different leaves in $L$ relate to each other. Using this characterization, we will determine that the leaves of $\Lambda^+(\lambda)$ ($\Lambda^-(\lambda)$, respectively) are exactly the stable chains (unstable chains, respectively) in the frontier of $\Omega_\lambda$. This motivates the following definitions:
\begin{defn}
For a leaf $\lambda\in \tilde{\F}$ with shadow $\Omega_\lambda$, let $\fr\Omega_\lambda^s$ be the closure in $\bar{\mathcal{O}}$ of the set of stable leaf slices in $\fr\Omega_\lambda$, and let $\Omega_\lambda^s\subseteq \mathcal{O}$ the disk bounded by the components of $\fr\Omega_\lambda^s$. We call a connected component of $\fr\Omega_\lambda^s$ a \textit{stable frontier chain} of $\Omega_\lambda$. We define $\fr\Omega_\lambda^u$, $\Omega_\lambda^u$, and \textit{unstable frontier chains} analogously, using the unstable leaf slices in $\fr\Omega_\lambda$.
\end{defn}
Observe, by definition we have $\Omega_\lambda\subseteq\Omega_\lambda^s$ and $\Omega_\lambda\subseteq\Omega_\lambda^u$, for each $\lambda \in L$. Additionally, by Theorem \ref{LMTthm2}, each stable and unstable frontier chain bounds a subset of a gap of $\phi_\lambda$ (a single gap may be bounded by both stable and unstable frontier components). See Figure \ref{fig:shadowfig}.

We seek to relate the shadows of $\mu\in L^+(\lambda)$ with the shadow of $\lambda$ ($\mu\in L^-(\lambda)$, respectively). We first give a lemma relating $\Omega_\lambda^s$ and $\Omega_\mu^s$ ($\Omega_\lambda^u$ and $\Omega_\mu^u$, respectively) for comparable $\lambda$ and $\mu$.
\begin{prop}\label{NEWshadowcontainment}
    For any $\lambda, \mu\in L$ such that $\lambda<\mu$, we have $\Omega_\mu^s\subseteq \Omega_\lambda^s$ and $\Omega_\lambda^u\subseteq \Omega_\mu^u$.
\end{prop}
\begin{proof}
    Consider any $\lambda, \mu\in L$ such that $\lambda<\mu$. Then, we have that $\lambda$ is contained in the connected component of $\widetilde{M}\setminus \mu$ below $\mu$ and $\mu$ is contained in the connected component of $\widetilde{M}\setminus \lambda$ above $\lambda$. Let $l^s\subseteq \fr\Omega_\lambda^s$ be a stable leaf slice. By Proposition \ref{fenleyprop}, $\lambda
    $ accumulates on $\tilde{l}^s$ going up with the flow, where $\tilde{l}^s$ is the lift of $l^s$ to $\widetilde{M}$. Since $L$ is co-oriented by the flow, it follows that $\tilde{l}^s$ is in the connected component of $\widetilde{M}\setminus \lambda$ below $\lambda$. Hence, $\mu$ can not intersect $\tilde{l}^s$ or any orbit on the other side of $\tilde{l}^s$. Therefore, $\Omega_\mu$ must be contained in the same connected component of $\mathcal{O}\setminus l^s$ as $\Omega_\lambda$. Repeating this for each $l^s\subseteq \fr\Omega_\lambda^s$, yields $\Omega_\mu\subseteq \Omega_\lambda^s$. Since $\Omega_\mu\subseteq \Omega_\lambda^s$ any frontier component of $\Omega_\mu$ is either  contained in $\fr\Omega_\lambda^s$ or $\Omega_\lambda^s$. This implies $\Omega_\mu^s\subseteq \Omega_\lambda^s$, as desired.
    
    An analogous argument using unstable leaf slices in $\fr\Omega_\mu^u$ holds, to conclude $\Omega_\lambda^u \subseteq \Omega_\mu^u$.
\end{proof}
From Proposition \ref{NEWshadowcontainment}, we can determine the position of two comparable leaves with respect to each other simply by analyzing the containment of their shadows. 

We now consider the case when $\lambda$ and $\mu $ are branching leaves. Since $\tilde{\varphi}$ is transverse to $\tilde{\F}$, no orbit of $\tilde{\varphi}$ can intersect both $\lambda$ and $\mu$. Thus, $\Omega_\lambda$ and $\Omega_\mu$ are disjoint in $\mathcal{O}$. We now provide a lemma determining which frontier components separate $\Omega_\lambda$ and $\Omega_\mu$.

\begin{lem}\label{branching2frontier}
    Suppose $\lambda$ and $\mu$ are positively branching leaves (negatively branching, respectively). Then there exists a stable (unstable, respectively) leaf slice $l$ such that $l\subset \fr\Omega_\lambda$ separating $\Omega_\lambda$ from $\Omega_\mu$. 
\end{lem}
\begin{proof}
    We first consider the case where $\lambda$ and $\mu$ are positively branching. The proof where $\lambda$ and $\mu$ are negatively branching is analogous, switching the roles of stable with unstable.

    By assumption $\lambda$ and $\mu$ are incomparable, so the shadows $\Omega_\lambda$ and $\Omega_\mu$ are disjoint. It follows from Proposition \ref{fenleyprop}, that there exists a leaf slice $l\subseteq \fr\Omega_\lambda$, separating $\Omega_\lambda$ and $\Omega_\mu$ in $\mathcal{O}$. We will now show that this $l$ is a stable leaf slice. 
    
    Since $\lambda$ and $\mu$ are positively branching, there exists $\eta\in L$ such that $\eta<\lambda,\mu$. By Proposition \ref{NEWshadowcontainment}, we have that $\Omega_\eta^u\subseteq \Omega_\lambda^u$ and $\Omega_\eta^u\subseteq \Omega_\mu^u$. Since $l$ is the frontier component which separates $\Omega_\lambda$ and $\Omega_\mu$, if $l$ was unstable, then $\Omega_\eta \subseteq \Omega_\lambda^u$ and $\Omega_\eta\subseteq\Omega_\mu^u$. But this is not possible since $l$ separates $\Omega_\lambda$ and $\Omega_\mu$, which is a contradiction. Therefore, $l$ must be a stable leaf slice.
\end{proof}

As a consequence of Lemma \ref{branching2frontier}, we have that for any positively branching leaves $\lambda$ and $\mu$ (negatively branching, respectively) and any leaf $\eta\in L^+(\mu)$ ($\eta\in L^-(\mu)$, respectively), $\Omega_\eta$ is separated from $\Omega_\lambda$ by the same frontier component separating $\Omega_\mu$ from $\Omega_\lambda$. 

\begin{cor}\label{incomparableShadows}
Suppose that $\lambda$ and $\mu$ are positively branching leaves (negatively branching, respectively) and let $l\subseteq \fr\Omega_\lambda$ separating $\Omega_\lambda$ and $\Omega_\mu$. Then for any $\eta\in L^+(\mu)$ ($\eta\in L^-(\mu)$, respectively), the shadow $\Omega_\eta$ lies on the same side of $l$ as $\Omega_\mu$.
\end{cor}
\begin{proof}
    We prove the claim when $\lambda$ and $\mu$ are positively branching. The proof of the claim when $\lambda$ and $\mu$ are negatively branching is analogous, swapping the roles of stable and unstable foliations.

    Since $\lambda$ and $\mu$ are positively branching, by Lemma \ref{branching2frontier} we have that there exists $l\subseteq \fr\Omega_\lambda$ and $l'\subseteq\fr\Omega_\mu$ stable leaf slices, each separating the shadows $\Omega_\lambda$ and $\Omega_\mu$. Let $\tilde{l}'$ denote the lift of $l'$ to $\widetilde{M}$. By Proposition \ref{fenleyprop}, we have that $\mu$ accumulates on $\tilde{l}'$ going up with respect to the flow. Hence $\tilde{l}'$ is contained in the component of $\widetilde{M}\setminus \mu$ below $\mu$. 

    By definition of $L^+(\mu)$, any $\eta\in L^+(\mu)$ is contained in the component of $\widetilde{M}\setminus \mu$ above $\mu$. Since $L$ is co-oriented by the flow, any such $\eta$ can not intersect $\tilde{l}'$ and has shadow $\Omega_\eta \subseteq \mathcal{O}\setminus l'$ contained in the same component as $\Omega_\mu$. Therefore, $\Omega_\eta$ is contained on the same side of $l$ as $\Omega_\mu$, as desired.
\end{proof}

Proposition \ref{NEWshadowcontainment} and Corollary \ref{incomparableShadows} will allow us to determine $\core(L^+(\lambda))$ and $\core(L^-(\lambda))$ as subsets of $\partial\mathcal{O}$. We first define some relevant subsets of $\partial \mathcal{O}$.
\begin{defn}
    Let $\lambda\in L$. We define $I_\lambda^s\subset \partial\mathcal{O}$ to be the open arcs bounded by stable frontier chains of $\fr\Omega_\lambda^s$. We define $I_\lambda^u\subset \partial\mathcal{O}$ analogously using unstable frontier chains of $\fr\Omega_\lambda^u$.
\end{defn}
See Figure \ref{fig:propfig} for an illustration of $I_\lambda^s$ and $I_\lambda^u$. Using Proposition \ref{NEWshadowcontainment} and Corollary \ref{incomparableShadows}, we will show that for any $\mu\in L^+(\lambda)$ ($\mu\in L^-(\lambda)$, respectively) $\core(\phi_\mu)$ is disjoint from $I_\lambda^s$ ($I_\lambda^u$, respectively). This gives the following lemma.
\begin{figure}[H]
    \centering
    \includegraphics[width= 3.5in]{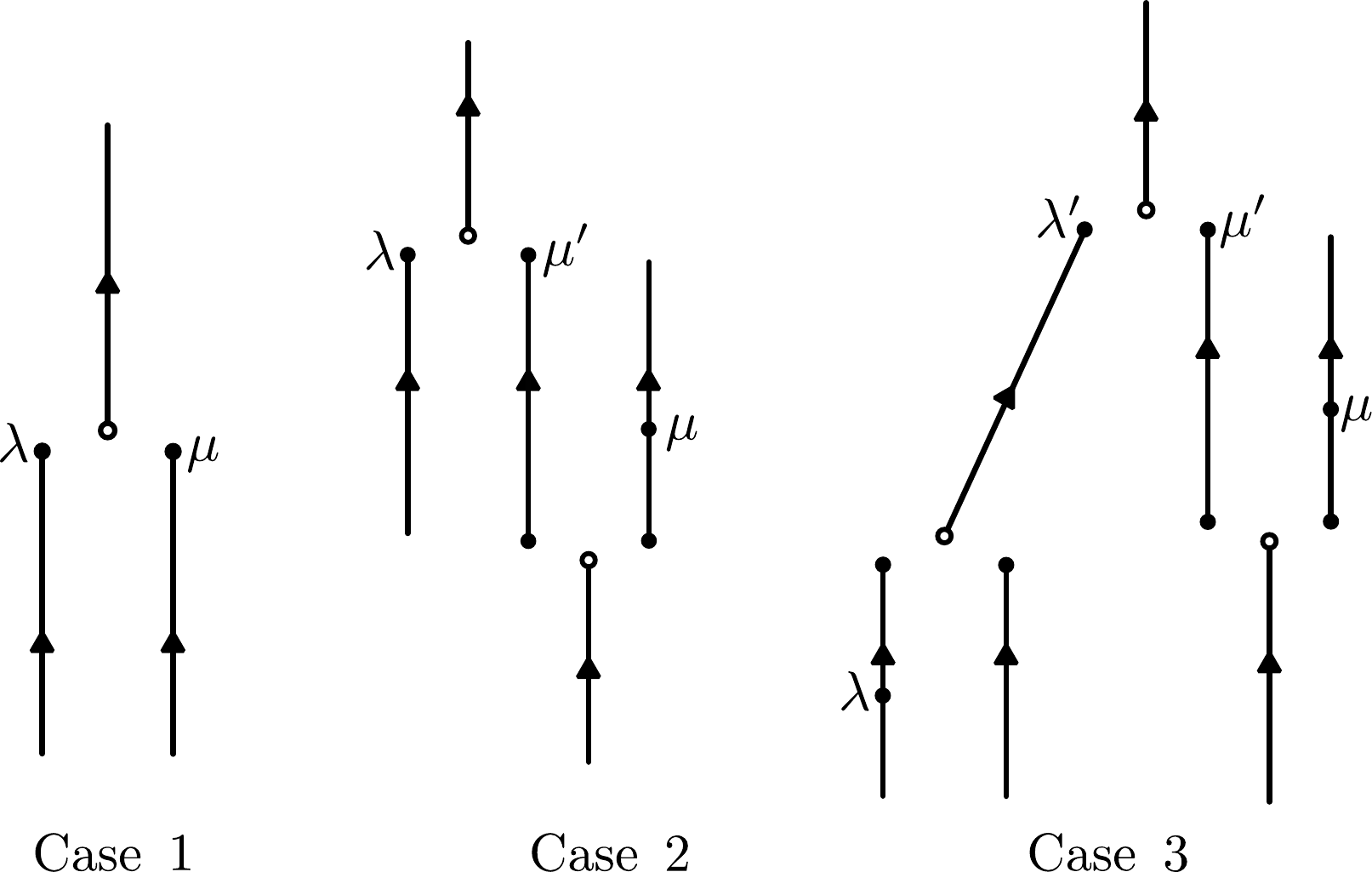}
    \caption{An illustration of the three cases for $\mu \in L^+(\lambda)$ in the proof of Lemma  \ref{coredisjointness}.}
    \label{fig:coredisjointness lemma}
\end{figure}
\begin{lem}\label{coredisjointness}
     For any $\lambda\in L$ we have $\core(L^+(\lambda)) \cap I_\lambda^s = \emptyset$  and $\core(L^-(\lambda)) \cap I_\lambda^u = \emptyset$.
\end{lem}
\begin{proof}
    We first prove that $\core(L^+(\lambda)) \cap I_\lambda^s = \emptyset$. The proof for $I_\lambda^u$ is the same replacing $(+)$ with $(-)$ and stable with unstable. We proceed by cases, showing that $\core(\phi_\mu)$ is disjoint from $I_\lambda^s$ first for  $\lambda<\mu$ and then for $\mu\in L^+(\lambda)$ incomparable with $\lambda$. 
    
    Suppose that $\mu\in L^+(\lambda)$ is comparable with $\lambda$. Then, by Proposition \ref{NEWshadowcontainment}, we have that $ \Omega_\mu^s\subseteq\Omega_\lambda^s$. Thus, $\core(\phi_\mu)$ is disjoint from $I_\lambda^s$, for all $\mu$ such that $\lambda<\mu$.

    We now consider $\mu\in L^+(\lambda)$ incomparable with $\lambda$. Then $\Omega_\mu$ is disjoint from $\Omega_\lambda$. There are three possible cases for such a $\mu$ and in each case we will show $\Omega_\mu$ is separated from $\Omega_\lambda$ by an unstable frontier component of $\Omega_\lambda$. This shows that $\core(\phi_\mu)\subseteq I_\lambda^u$ and therefore $\core(\phi_\mu)$ is disjoint from $I_\lambda^s$. 
    
    If $\mu$ is negatively branching with $\lambda$, then by Lemma \ref{branching2frontier} we have that $\Omega_\mu$ is separated from $\Omega_\lambda$ by an unstable frontier component and we are done. Similarly, if $\lambda$ is negatively branching with a leaf $\mu'$ such that $\mu \in L^-(\mu')$, then by Corollary \ref{incomparableShadows} we have that $\Omega_\mu$ is separated from $\Omega_\lambda$ by an unstable frontier component and we are done. Otherwise, there exists $\lambda'$ and $\mu'$ negatively branching such that $\lambda<\lambda'$ and either $\mu'=\mu$ or $\mu\in L^-(\mu')$. Then by Corollary \ref{incomparableShadows}, $\Omega_\mu$ is separated from $\Omega_{\lambda'}$ by an unstable frontier component $l\subseteq \fr\Omega_{\lambda'}$. By Proposition \ref{NEWshadowcontainment}, we have that $\Omega_\lambda^u\subseteq\Omega_{\lambda'}^u$ and $\Omega_{\lambda'}^s\subseteq\Omega_\lambda^s$. Since $l$ separating $\Omega_{\lambda'}$ and $\Omega_\mu$ is unstable, by the containments above, we conclude the frontier component $l'\subset\fr\Omega_\lambda$ separating $\Omega_{\lambda}$ from $\Omega_\mu$ can not be stable. Therefore, $l'$ is unstable and $\core(\phi_\mu)\subseteq I_\lambda^u$, which completes the proof.
\end{proof}

By Lemma \ref{coredisjointness}, we conclude that $\core(L^+(\lambda))\subseteq \partial\mathcal{O}\setminus I_\lambda^s$ and $\core(L^-(\lambda))\subseteq \partial\mathcal{O}\setminus I_\lambda^u$. In the next proposition, we show that when $\partial\mathcal{O}$ is a minimal universal circle we can upgrade this containment to an equality. 

\begin{figure}[h]
    \centering
    \includegraphics[width= 2.2in]{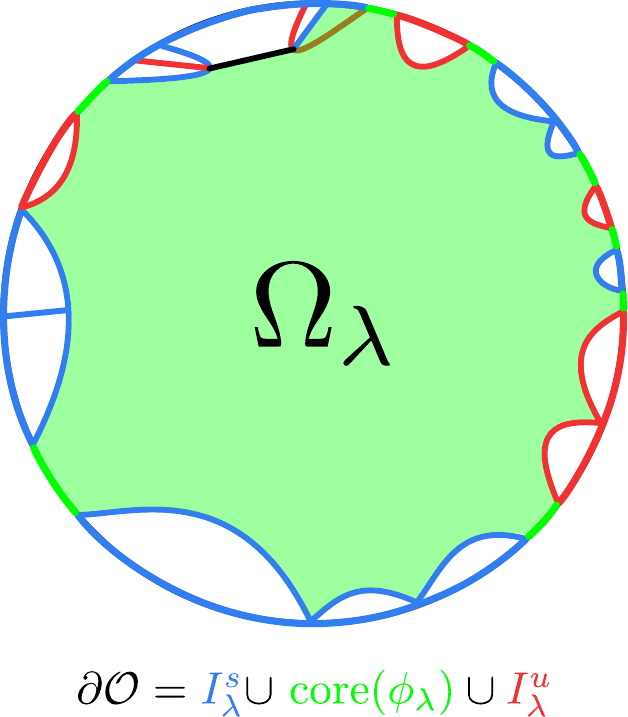}
    \caption{An illustration of Proposition \ref{core+-}. A shadow of a leaf $\lambda\in L$ with nonempty frontier. $\partial\mathcal{O}$ decomposes into the union of $I_\lambda^-$ (red), $\core(\phi_\lambda)$ (green), and $I_\lambda^+$ (blue).}
    \label{fig:propfig}
\end{figure}

One property we use to do so is the \textit{lower semicontinuity} of the family of subsets $\set{\core(\phi_\lambda)}_{\lambda\in L}$ {\cite[Lemma 2.2.6]{Calegari2006}}---that is, for any $x\in\core(\phi_\lambda)$ and any sequence of leaves $\lambda_i\to \lambda$ there exists a sequence $\set{x_i}$ such that $x_i\to x$ and $x_i\in\core(\phi_{\lambda_i})$, for each $i$. In particular, the lower semicontinuity of $\set{\core(\phi_\lambda)}_{\lambda\in L}$ implies $\core(\phi_\lambda)\subseteq \core(L^+(\lambda))$ and $\core(\phi_\lambda)\subseteq \core(L^-(\lambda))$.
\begin{prop}\label{core+-}
    Suppose that $\psi$ is a pseudo-Anosov flow on $M$ almost transverse to $\F$ and $(\partial \mathcal{O},\phi_\lambda)$ is a minimal universal circle for $\F$. For all $\lambda\in L$ we have $\core(L^+(\lambda))=\partial\mathcal{O}\setminus I_\lambda^s$ and $\core(L^-(\lambda))=\partial\mathcal{O}\setminus I_\lambda^u$.
\end{prop}
\begin{proof}
    We prove the claim by decomposing $\partial\mathcal{O}$ in two ways, to obtain the desired equality. Since $\partial\mathcal{O}$ is a minimal universal circle, we have that $\partial\mathcal{O} = \core(L)$. Therefore, we have that $\partial\mathcal{O} = \core(L^+(\lambda))\cup\core(L^-(\lambda))$. Moreover, we can decompose $\partial\mathcal{O}$ into the union of the core and gaps of a particular $\phi_\lambda$. Hence, by definition of $I_\lambda^s$ and $I_\lambda^u$, we have $\partial\mathcal{O} = I_\lambda^u\cup \core(\phi_\lambda)\cup I_\lambda^s$.
    
    We have $\core(L^+(\lambda))\subseteq \partial\mathcal{O}\setminus I_\lambda^s$ by Lemma \ref{coredisjointness}. We now conclude the reverse containment, $\partial\mathcal{O}\setminus I_\lambda^s \subseteq \core(L^+(\lambda))$. By lower semicontinuity of $\set{\core(\phi_\lambda)}_{\lambda\in L}$, we have $\core(\phi_\lambda)\subseteq \core(L^+(\lambda))$. By Lemma \ref{coredisjointness}, $I_\lambda^u$ is disjoint from $\core(L^-(\lambda))$. This implies $I_\lambda^u\subset\core(L^+(\lambda))$. Hence, $\partial\mathcal{O}\setminus I_\lambda^s = \core(\phi_\lambda)\cup I_\lambda^u \subseteq \core(L^+(\lambda))$, which completes the proof.
    
    The proof of $\core(L^-(\lambda))= \partial\mathcal{O}\setminus I_\lambda^u$ is analogous, swapping the roles of $(+)$ with $(-)$ and stable with unstable.
\end{proof}
From Proposition \ref{core+-}, it follows by definition of $\Lambda^\pm(\lambda)$ that the leaves of $\Lambda^+(\lambda)$ ($\Lambda^-(\lambda)$, respectively) are exactly the endpoints of stable frontier chains (unstable frontier chains, respectively).
\begin{cor}\label{orbit2calegari}
    Suppose that $\psi$ is a pseudo-Anosov flow on $M$ almost transverse to $\F$ and $(\partial \mathcal{O},\phi_\lambda)$ is a minimal universal circle for $\F$. Then for all $\lambda\in L$ we have $\Lambda^+(\lambda) = \set{\End(C)| C \subseteq \fr\Omega_\lambda^s \text{ a stable chain}}$ (respectively for $\Lambda^-(\lambda)$ using unstable chains $\fr\Omega^u$).
\end{cor}
It follows from this corollary that every leaf of $\Lambda^+_u$ is the pair of endpoints of a chain of stable leaves (leaves of $\Lambda^-_u$ from endpoints of chains of unstable leaves, respectively). In the proof of the main theorems, we will use Corollary \ref{orbit2calegari} to obtain a frontier chain in $\Lambda^\pm_u$ and use the $\pi_1(M)$ action on $\mathcal{O}$ to find a nearby regular leaf $l$ with endpoints $\End(l)\in \Lambda^\pm_u$.  We then conclude the main theorems by the density of $\pi_1(M)\cdot l$ and the closedness of $\Lambda^\pm_u$. In the following lemma we show how to find such a regular leaf.
\begin{lem}\label{noncornerleaf}
    Let $C\subseteq\fr\Omega_\lambda^s$ be a stable frontier chain ($C\subseteq \fr\Omega_\lambda^u$, respectively). Then there exists a stable regular leaf, $l$, (unstable, respectively) containing a noncorner periodic orbit, $o$, such that $\End(g^n\cdot C)\to \End(l)$, where $g\in stab(l)$. 
\end{lem}
\begin{proof}
    We prove the lemma for $C$ a stable frontier chain. Let $\End(C)=\set{c_1,c_2}$ such that the ideal arc $(c_1,c_2)\subset I_\lambda^s$. By the density of noncorner periodic orbits, there exists a noncorner periodic orbit $o$ contained in $\Omega_\lambda$ and contained in a stable regular leaf $l$. Moreover, we can choose $o$ close to $C$ such that the unstable leaf through $o$, $l^u$, intersects $C$. Let $\End(l)=\set{a_1,a_2}$ such that $(c_1,c_2)\subset (a_1,a_2)$ and let $\End(l^u)=\set{b_1,b_2}$ such that $b_1\in (c_1,c_2)$. Since $o$ is a noncorner periodic orbit, we can take $g\in stab(o)$ such that $a_1$ and $a_2$ are attracting fixed points and $b_1$ and $b_2$ are repelling fixed points under iteration by $g$. Then $g^n\cdot c_1 \to a_1$ and $g^n\cdot c_2\to a_2$ in $\partial\mathcal{O}$. Thus, we have $\End(g^n\cdot C)\to \End(l)$.
\end{proof}

Lemma \ref{noncornerleaf} proves that if there exists a $\lambda\in L$ such that $\fr\Omega_\lambda^s\neq \emptyset$ ($\fr\Omega_\lambda^u\neq \emptyset$, respectively), then there exists a stable (unstable, respectively) regular leaf with its endpoints in $\Lambda^+_u$ ($\Lambda^-_u$, respectively). We can now use the orbit of this regular leaf under the action of $\pi_1(M)$ to conclude that any leaf (or face of a leaf) with branching on one side has endpoints in $\Lambda^\pm_u$.

\begin{lem}\label{regularleafdyn}
    Suppose that $l$ is a stable regular leaf (unstable, respectively) such that $\End(l)\in \Lambda^+_u$ ($\End(l)\in \Lambda^-_u$, respectively). Then $\End(f)\in  \Lambda^+_u$ for any face $f$ of a stable leaf $l'$ without branching on both sides ($\End(f)\in  \Lambda^-_u$ for $f\subseteq l'$ unstable, respectively).
\end{lem}
\begin{proof}
We prove the claim in the case where $l$ is a stable leaf. The proof when $l$ is an unstable leaf is the same, replacing $(+)$ with $(-)$. Since $\psi$ is a transitive pseudo-Anosov flow, we have that the orbit $\pi_1(M)\cdot l$ is dense in $\mathcal{O}$. By Theorem \ref{CalegariLamination}, $\Lambda^+_u$ is $\pi_1(M)$-invariant and closed, so we have $\overline{\pi_1(M)\cdot\End(l)}\subset \Lambda^+_u$. 

If $l'\in \pi_1(M)\cdot l$, then $l'$ is also regular and we have proven the claim. Suppose that $f\subset l'$ for $l'$ not in $\pi_1(M)\cdot l$. Then there exists a sequence of leaves $\set{l_n}\subset\pi_1(M)\cdot l$ converging to $f$ in the Hausdorff topology on $\overline{\mathcal{O}}$. By assumption, $f$ has branching on only one side. Since $\pi_1(M)\cdot l$ is dense in $\mathcal{O}$, we can take $\set{l_n}$ to be on the side of $f$ containing no branching. 
Hence, $\set{l_n}$ converge only to $f$ in the Hausdorff topology on $\overline{\mathcal{O}}$. Since each $\End(l_n)\subset \Lambda^+_u$ and $\End(l_n)$ converges to $\End(f)$ in $\partial \mathcal{O}$, we conclude $\End(f)\in \Lambda^+_u$. This completes the proof of the claim.
\end{proof}

\subsection{Proof of Main Theorems}
We now prove the first of our main theorems. We first note that $\partial\mathcal{O}^{s,u}$ is closed, and hence forms a lamination of $\partial\mathcal{O}$, if and only if $\bar\F^{s,u}$ has Hausdorff leaf space \cite{BarthelmeBonattiMann2025}.
    
\begin{proof}[Proof of Theorem \ref{HausdorffThm}]
We  prove the theorem for $\bar\F^s_\psi$ Hausdorff. The argument for $\bar\F^u_\psi$ is analogous, switching the roles of $(+)$ with $(-)$ and $\bar\F^s_\psi$ with $\bar\F^u_\psi$. We first show that $\partial \mathcal{O}^s\subseteq \Lambda^+_u$. By assumption, there exists $\lambda\in L$ such that $\fr\Omega_\lambda$ contains a stable leaf slice. Since $\bar\F^s_\psi$ is Hausdorff, any stable frontier chain $C\subseteq \fr\Omega_\lambda^s$ consists of a single face of a leaf. By Corollary \ref{orbit2calegari} we have $\End(C)\in \Lambda^+_u$. Then, applying Lemma \ref{noncornerleaf} using $C$, we have that there exists a stable regular leaf $l$ such that $\End(l)\in \Lambda^+_u$. Since $\bar{\F}^s_\psi$ is Hausdorff, applying Lemma \ref{regularleafdyn} using $l$ yields that for any leaf $l'\in \bar{\F}^s_\psi$ and any face $f\subseteq l'$ we have $\End(f)\in \Lambda^+_u$. Therefore, we conclude $\partial\mathcal{O}^s\subseteq \Lambda^+_u$.

To complete the proof of this theorem we now show that $\Lambda^+_u\subseteq \partial\mathcal{O}^s$. Since $\bar\F^s_\psi$ is Hausdorff, every stable frontier chain $C\subseteq\fr\Omega_\lambda^s$ consists of a single leaf slice. Then, by Corollary \ref{orbit2calegari}, we have that $\Lambda^+(\lambda)\subset \partial\mathcal{O}^s$ for all $\lambda\in L$. Since $\bar\F^s_\psi$ is Hausdorff, $\partial\mathcal{O}^s$ is a lamination and hence closed. Therefore, $\Lambda^+_u = \overline{\bigcup_{\lambda\in L}\Lambda^+(\lambda)}\subseteq \partial\mathcal{O}^s$. 
\end{proof}

In the case where $\bar\F^s_\psi$ is non-Hausdorff ($\bar\F^u_\psi$, respectively), $\partial\mathcal{O}^s$ ($\partial\mathcal{O}^u$, respectively) is not closed. Therefore, we can not have equality as in the case of Theorem \ref{HausdorffThm}. In the case where $\F^s_\psi$ ($\F^u_\psi$, respectively) has leaves branching on one side only, we can recover one containment.

\begin{proof}[Proof of Theorem \ref{OneSided}]
    We  prove the theorem for $\bar\F^s_\psi$. The argument for $\bar\F^u_\psi$ is analogous, switching the roles of $(+)$ with $(-)$ and $\bar\F^s_\psi$ with $\bar\F^u_\psi$, and follows a similar argument to the proof of Theorem \ref{HausdorffThm}. 

    By assumption, there exists $\lambda\in L$ such that $\fr\Omega_\lambda$ contains a stable leaf slice. Hence there exists $C\subseteq \fr\Omega_\lambda^s$ a stable frontier chain. By Corollary \ref{orbit2calegari}, we have that $\End(C)\in \Lambda^+_u$. Then, applying Lemma \ref{noncornerleaf} using $C$, we have that there exists a regular leaf $l$ such that $\End(l)\in \Lambda^+_u$. Since every leaf of $\bar\F^s_\psi$ has branching on only one side, applying Lemma \ref{regularleafdyn} to $l$ yields that for any leaf $l'\in \bar{\F}^s_\psi$ and any face $f\subseteq l'$ we have $\End(f)\in \Lambda^+_u$. Therefore, we conclude $\partial\mathcal{O}^s\subset \Lambda^+_u$. 

    To see that the containment is strict we observe that since $\pi_1(M)\cdot l$ is dense in $\mathcal{O}$, any chain of leaves $C'$ which is the union of all branching leaves in a cataclysm of leaves, is contained in the orbit closure. Hence, $\End(C')\in \Lambda^+_u$, but $\End(C')\not\in \partial\mathcal{O}^s$. 
\end{proof}
It follows as a corollary that we can also obtain that the closures of $\partial\mathcal{O}^{s,u}$ are also contained in $\Lambda^\pm_u$. In Section \ref{sec:Example}, we give an example of a taut foliation and a transverse pseudo-Anosov flow with the property that $\overline{\partial\mathcal{O}^s}\subsetneq \Lambda^+_u$.
\begin{proof}[Proof of Corollary \ref{OneSidedCor}]
     By Theorem \ref{OneSided}, we have $\partial\mathcal{O}^s\subsetneq \Lambda^+_u$. Since $\Lambda^+_u$ is closed, $\Lambda^+_u$ contains the closure of $\partial\mathcal{O}^s$ in the space of distinct unordered pairs of points in $\partial\mathcal{O}$. By definition this closure is $\overline{\partial\mathcal{O}^s}$, so $\overline{\partial\mathcal{O}^s}\subseteq \Lambda^+_u$. The same proof holds to show $\overline{\partial\mathcal{O}^u}\subseteq \Lambda^-_u$.
\end{proof}
Finally, we prove Theorem \ref{TwoSided}--the case where $\bar\F^s_\psi$ ($\bar\F^u_\psi$, respectively) has leaves branching on both sides. The proof method used in Theorem \ref{OneSided} applies analogously in this case.
\begin{proof}[Proof of Theorem \ref{TwoSided}]
    We  prove the theorem for $\bar\F^s_\psi$. The argument for $\bar\F^u_\psi$ is analogous, switching the roles of $(+)$ with $(-)$ and $\bar\F^s_\psi$ with $\bar\F^u_\psi$. By assumption, there exists $\lambda\in L$ such that $\fr\Omega_\lambda$ contains a stable leaf slice. Hence there exists $C\subseteq \fr\Omega_\lambda^s$ a stable frontier chain. By Corollary \ref{orbit2calegari}, we have that $\End(C)\in \Lambda^+_u$. Then, applying Lemma \ref{noncornerleaf} to $C$, we have that there exists a regular leaf $l$ such that $\End(l)\in \Lambda^+_u$. Applying Lemma \ref{regularleafdyn} to $l$ yields that for any face $f\subseteq l'$ in a stable leaf with branching on one side, we have $\End(f)\in \Lambda^+_u$. By Theorem \ref{Fenleybranching}, there are only finitely many $\pi_1(M)$-orbits of stable leaves with branching on both sides. Therefore, we have that all  but finitely many $\pi_1(M)$-orbits of leaves in $\partial\mathcal{O}^s$ are contained in $\Lambda^+_u$, as desired.
\end{proof}

\section{Reconstructing $\Lambda^\pm_u$}\label{sec:Reconstruction}

In this section we provide a proof of Theorem \ref{reconstruct}. We show that for any action $\pi_1(M)\curvearrowright S^1$ coming from the orbit space of a pseudo-Anosov flow there is a finite set $S$ of lamination pairs such that any minimal orbit space universal circle with the same action defines a lamination pair $(\Lambda^+_u, \Lambda^-_u)\in S$. By a result of Barthelm\'e, Bonatti, and Mann \cite[Theorem B]{BarthelmeBonattiMann2025} if $\pi_1(M)\curvearrowright S^1$ comes from an orbit space of a pseudo-Anosov flow we can uniquely reconstruct the stable and unstable foliations of the associated pseudo-Anosov flow. Combining this result with the main theorems yields most of the leaves in $\Lambda^+_u$ and $\Lambda^-_u$. A class of leaves which are not described by the $\pi_1(M)$ action are \textit{diagonal leaves}:
\begin{defn}
    Let $C$ be a cataclysm of leaves in $\bar{\F}^s_\psi$ ($\bar{\F}^u_\psi$, respectively). A diagonal leaf in $\Lambda^+_u$ ($\Lambda^-_u$, respectively) is a leaf $\End(C')$, where $C'\subsetneq C$ and leaves in $C'$ share consecutive endpoints.
\end{defn}
See Figure \ref{fig:diagshadow} for an illustration of diagonal leaves. By Corollary \ref{orbit2calegari}, the leaves of $\Lambda^+(\lambda)$ and $\Lambda^-(\lambda)$ are the stable and unstable frontier chains of $\Omega_\lambda$, for each $\lambda$. This implies that the only leaves of $\Lambda^+_u$ and $\Lambda^-_u$ which are not contained in $\overline{\partial\mathcal{O}^s}$ and $\overline{\partial\mathcal{O}^u}$, respectively, are diagonal leaves. We prove Theorem \ref{reconstruct} by showing that given $\partial\mathcal{O}^s$ there are finitely many choices of diagonal leaves for $\Lambda^+_u$ and $\Lambda^-_u$. 
\begin{proof}[Proof of Theorem \ref{reconstruct}]
    We prove the claim for $\Lambda^+_u$. An analogous argument holds for $\Lambda^-_u$ replacing $\bar\F^s_\psi$ with $\bar\F^u_\psi$.
    
    Suppose $\rho:\pi_1(M)\to \Homeo^+(S^1)$ is an action coming from the orbit space of a pseudo-Anosov flow. Then by \cite[Theorem B]{BarthelmeBonattiMann2025} we can uniquely reconstruct $\bar{\F}^s_\psi$ from the action $\pi_1(M)\curvearrowright S^1$. If $\bar\F^s_\psi$ is Hausdorff, then by Theorem \ref{HausdorffThm}  we have  $\Lambda^+_u=\partial\mathcal{O}^s$. 
    
    Suppose $\bar\F^s_\psi$ is non-Hausdorff. By Theorem \ref{OneSided} and Theorem \ref{TwoSided} we have that every face of every leaf which is separated or branching on one side only has endpoints in $\Lambda^+_u$. Moreover, for any cataclysm $C$ in $\bar\F^s_\psi$, we have $\End(C)\in \overline{\partial\mathcal{O}^s}$ which implies $\End(C)\in \Lambda^+_u$.
    
    By Corollary \ref{orbit2calegari}, for all $\lambda\in L$, every leaf of $\Lambda^+(\lambda)$ is the endpoints of a stable frontier chain in $\fr\Omega_\lambda^s$. This implies that the only leaves in $\Lambda^+_u$ not in $\overline{\partial\mathcal{O}^s}$ are diagonal leaves. We now show that given the reconstructed ${\partial\mathcal{O}^s}$ there are finitely many choices of diagonal leaves for $\Lambda^+_u$ and hence finitely many laminations for $\Lambda^+_u$. 
    
    Let $C$ be a cataclysm of leaves in $\bar\F^s_\psi$ and let $P_C\subset\Lambda^+_u$ be the collection of leaves coming from each leaf in $C$ and $\End(C)$. Since every branching leaf has nontrivial stabilizer, by Theorem \ref{Fenleybranching} $P_C$ is a finite sided polygon. Therefore, there are finitely many choices of diagonal leaves in $P_C$. This choice determines the diagonal leaves for any representative in $\pi_1(M)\cdot P_C$. Hence there are finitely many choices of diagonal leaves for an orbit $\pi_1(M)\cdot P_C$. By Theorem \ref{Fenleybranching}, there are finitely many orbits of cataclysms in $\bar\F^s_\psi$. Since there are finitely many choices in each of these orbits, there are at most finitely many laminations for $\Lambda^+_u$ coming from the action $\rho$. 
    
    By repeating the same argument using $\partial\mathcal{O}^u$, there are finitely many laminations $\Lambda^-_u$ coming from $\rho$. Enumerating over all pairs $(\Lambda^+_u,\Lambda^-_u)$, we conclude there is a finite set $S$ of lamination pairs such that any minimal orbit space universal circle $S^1_u$ with action $\rho$ determines a lamination pair $(\Lambda^+_u,\Lambda^-_u)\in S$.
\end{proof}
\section{Sharpness of Main Theorems}\label{sec:Example}
The goal of this section is to construct an example of a pseudo-Anosov flow $\psi$ and a transverse taut foliation $\F$ on an atoroidal three-manifold $M$ such that $\overline{\partial\mathcal{O}^s}\subsetneq \Lambda^+_u$. The construction provided in this section is flexible and can be used to obtain a family of pseudo-Anosov flows and foliations on distinct hyperbolic three-manifolds, exhibiting the same property. As argued in the proof of Theorem \ref{reconstruct}, Corollary \ref{orbit2calegari} implies that the only leaves in $\Lambda^+_u$ which are not in $\partial\mathcal{O}^s$ are diagonal leaves. Therefore, we construct $\psi$ and $\F$ such that $\Lambda^+_u$ contains a diagonal leaf.

\begin{figure}[h]
    \centering
    \includegraphics[width= 6in]{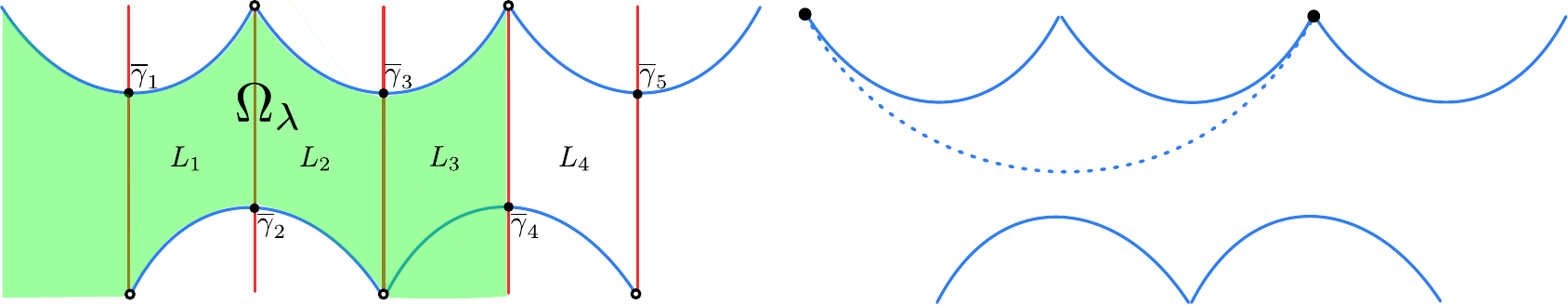}
    \caption{Left: The shadow of $\Omega_\lambda$ contains exactly three lozenges from a line of four lozenges. Since $\fr\Omega_\lambda$ contains a stable frontier chain (top) consisting of exactly two of three branching leaves, $\Lambda^+_u$ contains a diagonal leaf. Right: An illustration of a diagonal leaf (dotted) in $\Lambda^+_u$ spanning two of three leaves in a cataclysm of branching leaves}
    \label{fig:diagshadow}
\end{figure}

\begin{thm}\label{examplethm}
    There exists a closed atoroidal three-manifold $M$ admitting a pseudo-Anosov flow $\psi$ and taut foliation $\F$ such that $\psi$ and $\F$ are transverse, $\partial \mathcal{O}$ is a minimal universal circle for $\F$, and $\Lambda^+_u$ contains a diagonal leaf.
\end{thm}

The approach used to construct $M$, $\psi$, and $\F$ was suggested to the author by Sam Taylor and Michael Landry. The proof of Theorem \ref{examplethm} follows a special case of a construction introduced by Gabai and Mosher, which is outlined in \cite{MR1184414}. This construction is fully developed in upcoming work of Michael Landry and Chi Cheuk Tsang \cite{LandryTsangPrep}. The construction relies on a building block called a \textit{round handle}, which contains a semiflow and a transverse foliation. By gluing a collection of round handles appropriately, one can arrange for certain chains of lozenges $\set{L_i}_{i=1,\ldots,n}$ in the orbit space of $\psi$. Furthermore, one can produce a leaf $\lambda\in L$ whose shadow contains a strict subchain of $\set{L_i}$. We show that $\psi$ and $\F$ can be constructed such that $\mathcal{O}$ contains a line of four lozenges $\set{L_i}_{i=1,\ldots,4}$ sharing sides in $\bar\F_\psi^u$ and such that there exists a leaf $\lambda$ whose shadow contains $L_i\subset\Omega_\lambda$, for $i=1,2,3$, and is disjoint from $L_4$. Such a $\psi$ and $\F$ give rise to a diagonal leaf in $\Lambda^+_u$. Therefore,  $\overline{\partial\mathcal{O}^s}\subsetneq \Lambda^+_u$ as desired.

We first introduce round handles and then proceed with the proof of Theorem \ref{examplethm}.
\subsection{Round Handles}
Consider the flow $\psi_0$ on $\R^3$ given by $\psi_0^t (x,y,z) = (\lambda^tx,\lambda^{-t}y,z+t)$. The $\Z$-action on $\R^3$ by translation in the $z$-coordinate is free, properly discontinuous, cocompact, and commutes with $\psi_0$. Let $C= \set{\pm 1}\times \set{\pm 1} \times \R$ and let $U_i$ be a small tubular neighborhood around a connected component $C_i\subset C$ such that $\partial(U\cap [-1,1]^2\times \R)$ is everywhere tangent to $\psi_0$, where $U=\cup_i U_i$. Therefore, $H:=\left(([-1,1]^2\times \R)\setminus U\right) / \Z$ is a solid torus and $\psi_0$ descends to a semiflow on $H$, denoted by $\psi_H$. 
\begin{defn}
    The solid torus $H$ equipped with the semiflow $\psi_H$ is called a \textit{round handle}.
\end{defn} 
We now note some dynamical properties of $\psi_H$. Observe that $\partial H$ decomposes into eight annuli, four annuli where $\psi_H$ is transverse to $\partial H$ and four annuli where $\psi_H$ is tangent to $\partial H$. Moreover, $\psi_H$ is inward pointing on two of the annuli and is outward pointing on the other two. The collection of inward and outward pointing annuli are denoted by $R^-(H)$ and $R^+(H)$, respectively, and the collection of annuli where $\psi_H$ is tangent is denoted by $A(H)$. On each annulus of $A(H)$, $\psi_H$ restricts to the product flow. Finally, observe that $\psi_H$ has a single hyperbolic periodic orbit, $\gamma$, contained in the interior of $H$, which is homotopic to the core of $H$. Let $W^s(\gamma)$ and $W^u(\gamma)$ denote the stable and unstable leaves of $\gamma$. By construction, $W^s(\gamma)$ intersects $R^-(H)$ while $W^u(\gamma)$ intersects $R^+(H)$. All other orbits of $\psi_H$ intersect a component of $R^+(H)$ and a component of $R^-(H)$ exactly once.

We further equip a round handle $H$ with a codimension-one foliation, $\F_H$, by a ``stack of chairs". For the construction of a foliation by a stack of chairs we refer the reader to {\cite[Example 4.19]{Calegari2007}} or {\cite[Example 5.2]{LandryTsang2025}}. Each leaf of $\F_H$ is a disk whose boundary accumulates on $R^+(H)$ and $R^-(H)$, and intersects each component of $A(H)$. By construction, $\psi_H$ is transverse to $\F_H$. See Figure \ref{fig:roundhandle} for an illustration of $H$ equipped with $\psi_H$ and $\F_H$.

\begin{figure}[H]
    \centering
    \includegraphics[width= 3.8in]{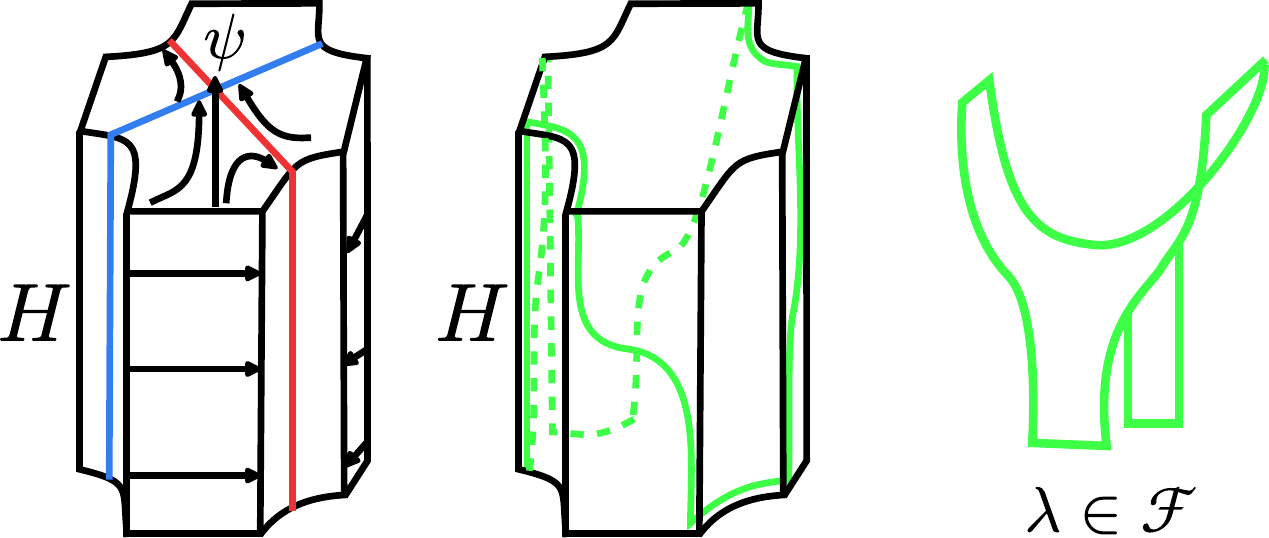}
    \caption{A round handle with its associated semiflow $\psi_H$ and foliation $\F_H$ by a stack of chairs.}
    \label{fig:roundhandle}
\end{figure}

\subsection{Proof of Theorem \ref{examplethm}}
Before proceeding to the proof of Theorem \ref{examplethm}, we prove a lemma which allows us to conclude that a lozenge is contained in the shadow of a leaf given that one of its sides is contained in the frontier. Note that we adopt the convention that a lozenge is an open rectangle in $\mathcal{O}$ (see {\cite[Section 2.4]{BarthelmeMann2026}} for an exposition on the theory of lozenges).

\begin{lem}\label{lozengelemma}
    Suppose that $\gamma$ and $\gamma'$ are the corners of a lozenge $L$. Let $l^s$ and $l^u$ be the stable and unstable leaves through $\gamma$ and let $r^s$ and $r^u$ be the stable and unstable rays in $l^s$ and $l^u$, respectively, which are sides of $L$. If $l^s\subset \fr\Omega_\lambda$ and $r^u\subset \Omega_\lambda$ for some $\lambda$, then $L\subset \Omega_\lambda$. Similarly, if $l^u\subset \fr\Omega_\lambda$ and $r^s\subset \Omega_\lambda$ for some $\lambda$, then $L\subset \Omega_\lambda$
\end{lem}

\begin{figure}[H]
    \centering
    \includegraphics[width= 3in]{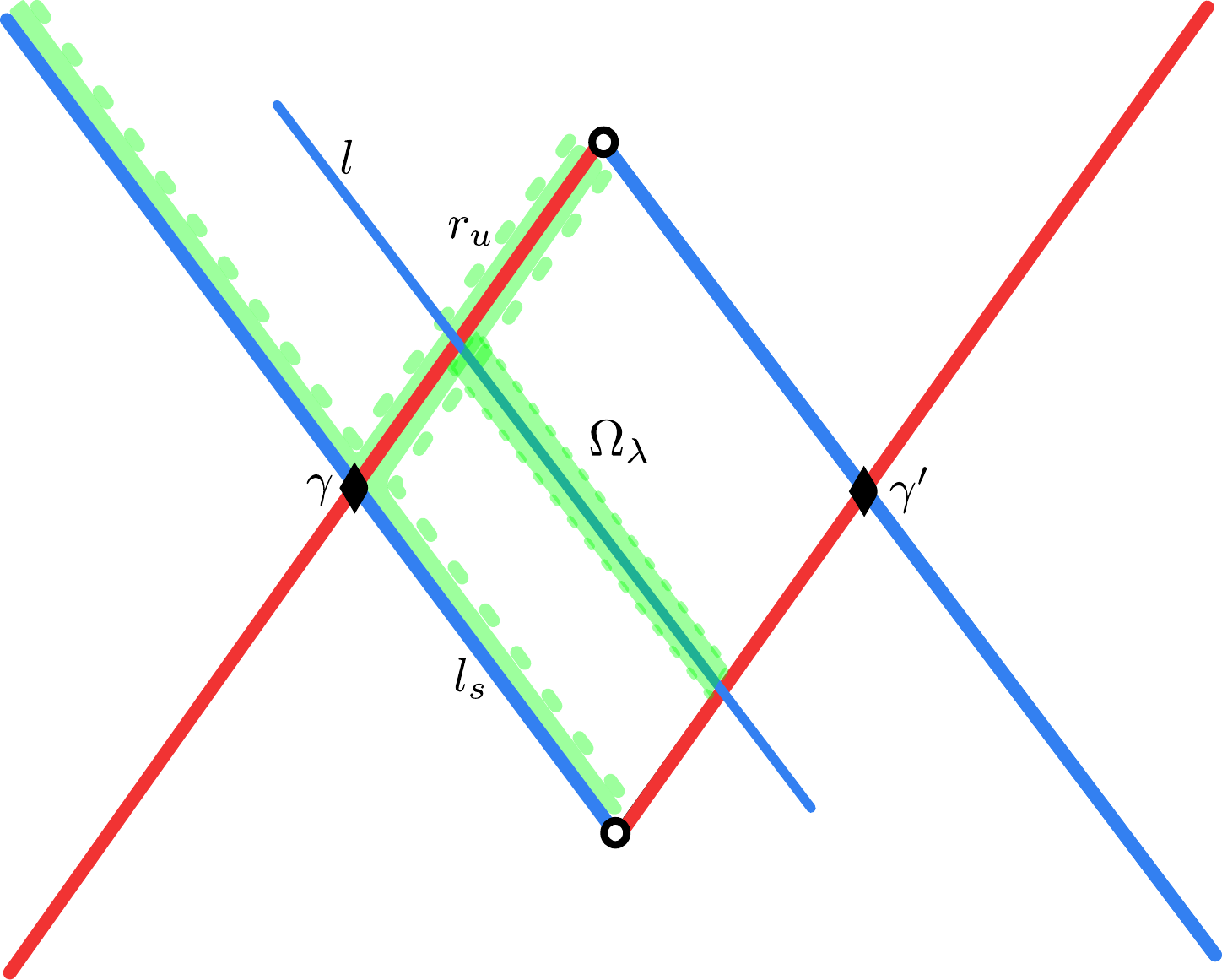}
    \caption{An illustration of the argument in the proof of Lemma \ref{lozengelemma}.}
    \label{fig:lozengelemma}
\end{figure}

\begin{proof}
    We prove the lemma in the case that $l^s\subset \fr\Omega_\lambda$ and $r^u\subset \Omega_\lambda$ for some $\lambda$. The proof in the case of $l^u\subset \fr\Omega_\lambda$ and $r^s\subset \Omega_\lambda$ is analogous.

    By assumption $l^s\subset \fr\Omega_\lambda$. This implies any unstable leaf intersecting $l^s$ can not be contained in $\fr\Omega_\lambda$. Since the interior of a lozenge is trivially foliated {\cite[Remark 2.4.6]{BarthelmeMann2026}} and $r^u\subset\Omega_\lambda$, for any stable leaf $l$ such that $l \cap r^u\neq \emptyset$, we must have $L\cap l\subset \Omega_\lambda$. Doing this for every such $l$ implies that $L\subset \Omega_\lambda$.
\end{proof}

We now proceed to the proof of Theorem \ref{examplethm}.

\begin{proof}[Proof of Theorem \ref{examplethm}]
The strategy of the proof is as follows. First, we construct a solid torus block $B$ obtained by gluing five round handles. The block $B$ contains a semiflow $\psi_B$ and transverse foliation $\F_B$. The flow $\psi_B$ contains five hyperbolic periodic orbits freely homotopic to the core of $B$, and the foliation $\F_B$ contains four leaves $\lambda_j$, $j=1,\ldots, 4$, intersecting the stable and unstable leaves of $\gamma_i$. Following a construction introduced by Gabai \cite{Gabai} and Mosher \cite{Mosher1992b}, via a series of gluings which preserve $\psi_B$ and $\F_B$, we obtain a closed atoroidal three-manifold $M$ with a pseudo-Anosov flow, $\psi$, and a transverse foliation, $\F$. Then, using Lemma \ref{lozengelemma} and combinatorics of the plane, we prove that $\mathcal{O}$ contains a line of exactly four lozenges sharing sides in $\bar\F^u_\psi$ and that there exists a leaf $\tilde{\lambda}_j\in \tilde{\F}$ whose shadow contains $L_1,L_2,$ and $L_3$ and is disjoint from $L_4$. Hence, $\fr\Omega_{\tilde{\lambda}_j}$ contains two branching stable leaves from a cataclysm containing exactly three leaves. By Corollary \ref{orbit2calegari}, $\Lambda^+_u$ contains a diagonal leaf and $\overline{\partial\mathcal{O}^s}\subsetneq \Lambda^+_u$. See Figure \ref{fig:diagshadow} for an illustration of the line of lozenges $\set{L_i}$ and the shadow $\Omega_{\tilde{\lambda}_j}$.

\textbf{\underline{Constructing $B$}:}
We now build the solid torus block $B$. Let $H_1,\ldots,H_5$ be a collection of round handles with semiflows $\psi_i$ and transverse foliations $\F_i$ by ``stack of chairs". Moreover, for $i$ even, we orient the flow lines of $\psi_i$ with the time reversed orientation, for $i=1,\ldots,5$. To construct the solid torus $B$ with the desired properties, we glue an annulus of $R^+(H_i)$ to an annulus of $R^-(H_{i+1})$ for each $i=1,\ldots,4$. In particular, let $R_i^\pm\subset R^\pm(H_i)$ be an annulus oriented by $\psi_i$ and let $a_i^u:=W^u(\gamma_i)\cap R_i^+$ and $a_{i}^s=W^s(\gamma_{i}) \cap R_{i}^-$ be oriented curves with the same orientation as $\gamma_i$. The block $B$ is constructed by gluing $H_i$ to $H_{i+1}$ via a homeomorphism $\phi_i:R_i^+\to R_{i+1}^-$ such that $\phi_i(a_i^u)$ is disjoint from $a^s_{i+1}$ and $\phi_i(a_i^u)$ lies to the left of $a_{i+1}^s$ for $i$ odd and to the right of $a_{i+1}^s$ for $i$ even, for $i=1,\ldots,4$. See Figure \ref{fig:blockfig} for an illustration of the construction of $B$.

The result of the sequence of gluings above is a solid torus $B$ with a semiflow, denoted $\psi_B$, where $\psi_B$ restricts to $\psi_i$ on the inclusion of $H_i$ in $B$. Further, by our choice of $\phi_i$, we have that $W^u(\gamma_i)$ and $W^s(\gamma_i)$ extend in $B$, for each $i=1,\ldots,5$. Hence, each $\gamma_i$ is a hyperbolic periodic orbit of $\psi_B$, for each $i=1,\ldots,5$. 

Observe that $\partial B$ decomposes into the union of six inward pointing and six outward pointing annuli, denoted $R^-(B)$ and $R^+(B)$ respectively, and 12 annuli where $\psi_B$ restricts to the product flow, denoted by $A(B)$. By our choice of $\phi_i$, all $W^u(\gamma_i)$ intersect a common outward pointing annulus in $\partial B$, and $W^s(\gamma_i)$ and $W^s(\gamma_{i+1})$ intersect a common inward pointing annulus for each $i$. 

Since each $H_i$ is equipped with a foliation $\F_i$ where every leaf accumulate on $R^+(H_i)$ and $R^-(H_i)$, by construction, $B$ inherits a foliation $\F_B$. The foliation $\F_B$ restricts to $\F_i$ on the inclusion of $H_i$ in $B$ and contains four annular leaves, $\lambda_i$, coming from the inclusion of $R^+_i$, for $i=1,\ldots, 4$. We now make the following observation regarding the intersection of $\lambda_i$ with the stable and unstable annuli:
\begin{obs}\label{constructionobs}
    It follows from the construction of $B$ that $\lambda_i$ intersects $W^u(\gamma_k)$ and $W^s(\gamma_{i+1})$ away from each of the periodic orbits, for all $k\leq i$ and $i=1,\ldots, 4$. 
\end{obs}
See Figure \ref{fig:blockfig} for an illustration of the block $B$ and the intersections noted in Observation \ref{constructionobs}.

\begin{figure}[H]
    \centering
    \includegraphics[width= 5in]{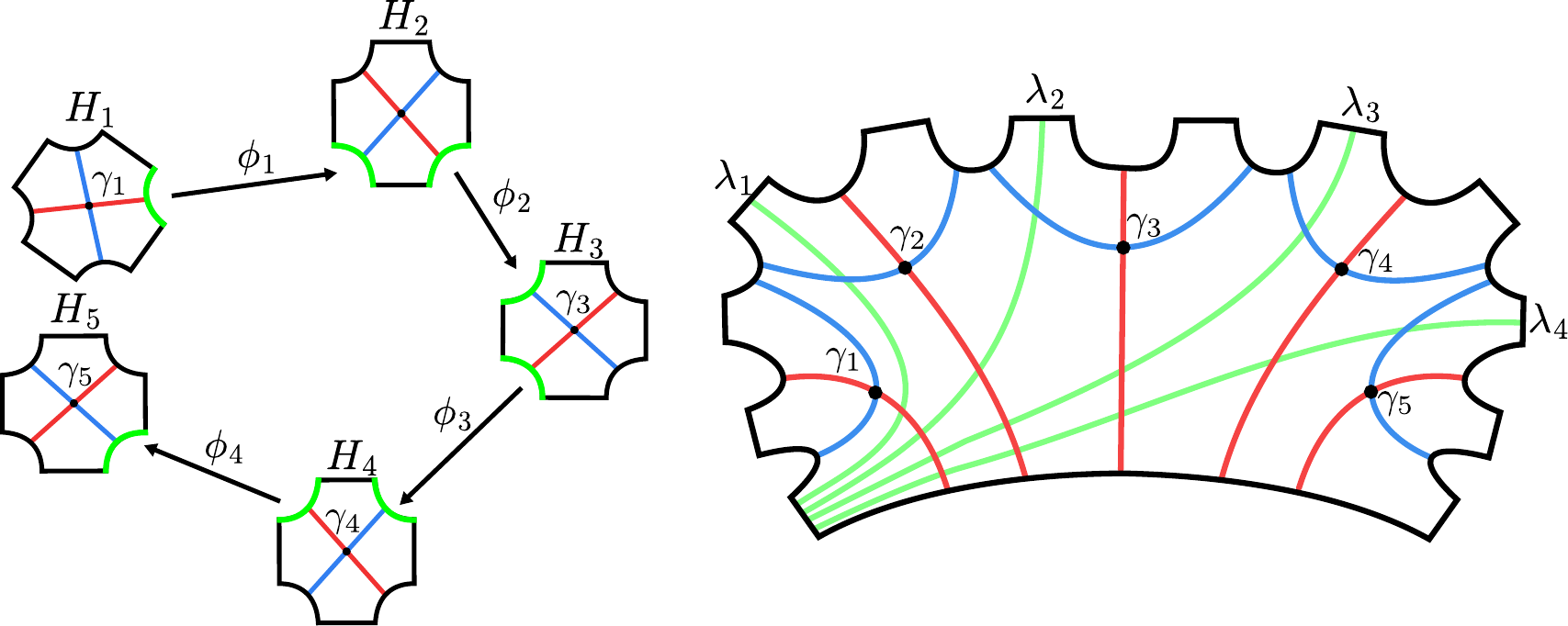}
    \caption{Left: A collection of five round handles $H_1,\ldots, H_5$ where $R^+_i$ is glued to $R^-_{i+1}$ via $\phi_i$. Right: A cross section of the solid torus $B$ with the intersection pattern of stable and unstable annuli with $\partial B$ and $\lambda_1,\ldots, \lambda_4$.}
    \label{fig:blockfig}
\end{figure}

\textbf{\underline{Constructing $N$}:}
Using the block $B$, we construct a compact three-manifold $N$ which will be used to obtain the desired atoroidal three-manifold $M$. The manifold $N$ is obtained by gluing $B$ to $S\times I$, for $S$ a compact surface. We now describe the choice of $S$, and a semiflow and transverse foliation on $S\times I$.

Let $S$ be a compact oriented surface with 12 boundary components and equip $S\times I$ with the product flow. We fix an identification of $A(B)$ with $\partial S\times I$ which maps oriented flow lines of $\psi_B|_{A(B)}$ onto product flow lines on $\partial S\times I$. To define a foliation $\F_S$ on $S\times I$ transverse to the product flow, we recall that a horizontal foliation on $S\times I$ (i.e. a foliation transverse to the $I$ fibers) is determined up to isomorphism by its holonomy representation $\rho:\pi_1(S)\to\Homeo^+(I)$ (see {\cite[Theorem 4.3]{Calegari2007}}). Therefore, let $\rho_S:\pi_1(S)\to \Homeo^+(I)$ be a representation such that $\rho_S|_{\partial S \times I}$ agrees with $\rho_B|_{A(B)}$, where $\rho_B:\pi_1(B)\to \Homeo^+(I)$ is the holonomy representation for $\F_B$. Then, $\rho_S$ determines a horizontal foliation $\F_S$ on $S\times I$, up to isomorphism, which agrees with $\F_B|_{A(B)}$ on $\partial S\times I$, by construction.

Let $N$ be the three-manifold obtained from $B$ and $S\times I$ by gluing $A(B)$ to $\partial S\times I$ via a homeomorphism. We require the gluing to map oriented flow lines of $\psi_B$ to oriented flow lines of the product flow on $S\times I$, and glue leaves of $\F_B|_{A(B)}$ to leaves of $\F_S|_{\partial S\times I}$. The result of this gluing is a compact three-manifold with two homeomorphic boundary components, $R^-(N)$ and $R^+(N)$, and a well-defined semiflow, $\psi_N$, and transverse foliation $\F_N$. Both $R^+(N)$ and $R^-(N)$ are leaves of $\F_N$ and $\psi_N$ points inwards along $R^-(N)$ and outwards along $R^+(N)$. Moreover, each $\lambda_i$ in $\F_B$ extends to a leaf in $\F_N$, which we continue to denote by $\lambda_i$, contained in the interior of $N$. By construction, the five hyperbolic periodic orbits of $\psi_B$ are also hyperbolic periodic orbits of $\psi_N$ and are contained in the interior of $N$. The stable and unstable leaves of $\gamma_i$ intersect $R^-(N)$ and $R^+(N)$, respectively, along a pair of essential closed curves. We label these closed curves $c^s(\gamma_i)_j\subset R^-(N)$ and $c^u(\gamma_i)_j\subset R^+(N)$, for $j=1,2$. 

\textbf{\underline{Constructing $M$}:}
We now construct the manifold $M$ and its associated flow $\psi$ and transverse foliation $\F$. To construct a manifold $M$ from $N$ we specify a gluing map $f:R_+(N)\to R_-(N)$. We choose $f$ such that $\set{f(c^u(\gamma_i)_j)}\cup \set{c^s(\gamma_i)_j}$ \textit{binds} $R_-(N)$ and \textit{efficiently intersects}, properties which we now define. The collection of curves $\set{f(c^u(\gamma_i)_j)}\cup \set{c^s(\gamma_i)_j}$ is said to bind $R_-(N)$ if any essential closed curve in $R^-(N)$ intersects a curve in $\set{f(c^u(\gamma_i)_j)}\cup \set{c^s(\gamma_i)_j}$. Additionally, the collection of curves $\set{f(c^u(\gamma_i)_j)}\cup \set{c^s(\gamma_i)_j}$ is said to efficiently intersect if no arc contained in $\set{f(c^u(\gamma_i)_j)}$ is path homotopic to an arc in $\set{c^s(\gamma_i)_j}$. Finally, we require that none of the curves in $\set{f(c^u(\gamma_i)_{j})}$ are freely homotopic to curves in $\set{c^s(\gamma_i)_{j}}$.

Let $M$ be the manifold obtained by gluing the boundary of $N$ via a map $f: R_+(N)\to R_-(N)$. Then, by our choice of $f$, \cite{MR1184414} and {\cite[Comment 2]{Mosher1992b}}, the resulting $M$ is atoroidal. Additionally, $M$ inherits a flow, $\psi$, and transverse foliation, $\F$. By Gabai \cite{Gabai} (see \cite[Proposition 4.3]{Mosher1992b} for a theorem statement), if necessary, $f$ can be isotoped such that the resulting flow $\psi$ is pseudo-Anosov. This yields the desired $M$, $\psi$, and $\F$. We now show that the periodic orbits $\gamma_i$ lift to corners of a line of lozenges $\set{L_i}_{i=1,\ldots,4}$, and that $L_1,L_2,L_3\subset \Omega_{\tilde{\lambda}_3}$, but $\Omega_{\tilde{\lambda}_3}$ is disjoint from $L_4$.

\textbf{\underline{Properties of $\mathcal{O}$ and $\Omega_{\tilde{\lambda}_j}$}:}
Let $\tilde{\gamma}_i$ and $\tilde{\lambda}_j$ be a lift of $\gamma_i$ and $\lambda_j$ in $\widetilde{M}$ and let $\bar{\gamma}_i$ be the projection of $\tilde{\gamma}_i$ to $\mathcal{O}$. Since $\gamma_1,\ldots, \gamma_5$ are a collection of freely homotopic periodic orbits, by \cite{Fen95b}, $\bar{\gamma}_1,\ldots, \bar{\gamma}_5$ are corners in a chain of lozenges. Moreover, we claim that no other periodic orbit of $\psi$ is in the same free homotopy class as $\gamma_i$. To see this, we observe that every periodic orbit of $\psi$, other than $\gamma_1,\ldots,\gamma_5$, intersects the quotient surface obtained from gluing $R^+(N)$ to $R^-(N)$ in $M$, with nonzero algebraic intersection. But, by construction, $\gamma_1,\ldots, \gamma_5$ have trivial algebraic intersection with this surface. Thus, no other periodic orbit of $\psi$ lies in the same free homotopy class as $\gamma_1,\ldots,\gamma_5$.

We now show that the chain of lozenges containing $\bar{\gamma}_i$ is a line consisting of exactly four lozenges. Let ${l_i^s}$ and $l^u_i$ denote the stable and unstable leaves through $\bar{\gamma}_i$. It follows from Observation \ref{constructionobs} that a half leaf of $l_k^u$ and a half leaf of $l^s_{j+1}$ are contained in $\Omega_{\tilde\lambda_j}$, and $l_k^s$ and $l^u_{j+1}$ are contained in the frontier of $\Omega_{\tilde\lambda_j}$, for all $k\leq j$ and each $j=1,\ldots,4$. See Figure \ref{fig:blockfig}.

Consider the shadow of $\tilde\lambda_1$. By Observation \ref{constructionobs}, we have that $l_1^s,l^u_2\subset\fr\Omega_{\tilde{\lambda}_1}$, and that an unstable half leaf based at $\bar\gamma_1$ and a stable half leaf based at $\bar\gamma_2$ are contained in $\Omega_{\tilde{\lambda}_1}$. By Lemma \ref{lozengelemma}, the lozenge with corner $\bar\gamma_1$ and the lozenge with corner $\bar\gamma_2$ are contained in $\Omega_{\tilde{\lambda}_1}$. Since $\Omega_{\tilde\lambda_1}$ is an open disk in $\mathcal{O}$ and $\bar\gamma_1$ and $\bar\gamma_2$ are connected by a chain of lozenges, if $\bar\gamma_1$ and $\bar\gamma_2$ are not the corners of the same lozenge, $\Omega_{\tilde\lambda_1}$ must contain some $\bar\gamma_i$. But $\lambda_1$ is disjoint from $\gamma_i$, for each $i$, so $\bar\gamma_1$ and $\bar\gamma_2$ must be corners of a single lozenge. Denote this lozenge by $L_1$. See Figure \ref{fig:lineoflozenge}.

\begin{figure}[H]
    \centering
    \includegraphics[width= 5in]{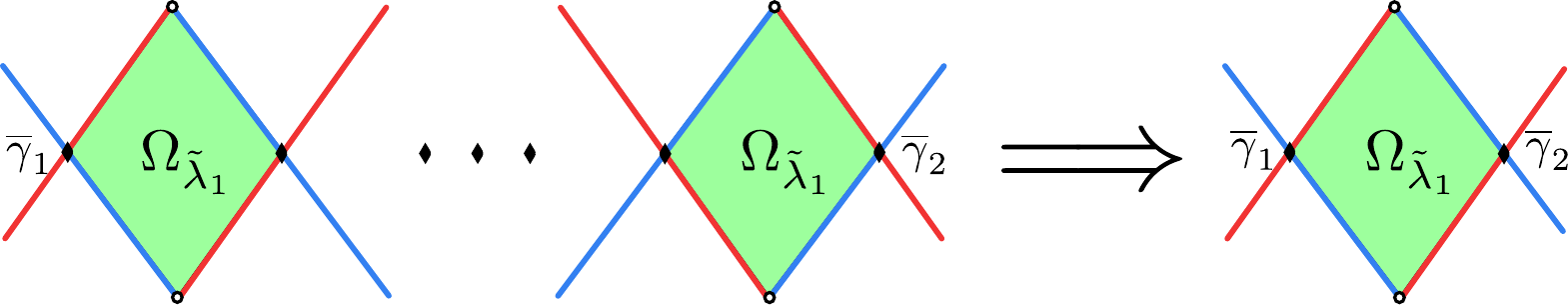}
    \caption{The shadow of $\Omega_{\tilde{\lambda}_1}$ contains the lozenges with corners $\bar{\gamma}_1$ and $\bar{\gamma}_2$ (left). Since $\lambda_1$ is disjoint from the periodic orbits $\gamma_1$ and $\gamma_2$ and $\Omega_{\tilde{\lambda}_1}$ is an open disk, $\bar{\gamma}_1$ and $\bar{\gamma}_2$ are corners of the same lozenge (right).}
    \label{fig:lineoflozenge}
\end{figure}

Similarly, in the case of $\tilde\lambda_2$, by Observation \ref{constructionobs} we have that there exist unstable half leaves based at $\bar\gamma_1$ and $\bar{\gamma}_2$ and a stable half leaf through $\bar\gamma_3$ which are contained in $\Omega_{\tilde{\lambda}_2}$, and $l^s_1,l^s_2,l^u_3\subset\fr\Omega_{\tilde\lambda_2}$. By Lemma \ref{lozengelemma}, $L_1$ and a lozenge with corner $\gamma_3$ are contained in $\Omega_{\tilde{\lambda}_2}$. Since $l^s_1,l^s_2\subset \fr\Omega_{\tilde{\lambda}_2}$ and $L_1\subset \Omega_{\tilde{\lambda}_2}$ we have that  $\bar\gamma_3$ is contained in the disk bounded by $l^s_1$ and $l^s_2$. Recall, $\bar\gamma_3$ is the corner of a chain of lozenges containing $\bar\gamma_1$ and $\bar\gamma_2$. Since $\Omega_{\tilde{\lambda}_2}$ is an open disk and $\lambda_2$ is disjoint from $\gamma_4$ and $\gamma_5$, it follows that $\bar\gamma_3$ is the corner of a lozenge with $\bar\gamma_1$ or $\bar\gamma_2$. By construction of $B$, we have that $\gamma_1$ and $\gamma_3$ have the same orientation. This implies that $\bar\gamma_2$ and $\bar\gamma_3$ are corners of a single lozenge. Denote this lozenge by $L_2$. Hence, $L_1$ and $L_2$ are adjacent lozenges sharing side $l^u_2$.

By repeating a similar argument using $\tilde\lambda_3$ and $\tilde\lambda_4$, we conclude that $\bar\gamma_1, \ldots,\bar\gamma_5$ are the corners of a line of lozenges $L_1,\ldots, L_4$ sharing unstable sides, where $\bar\gamma_j$ and $\bar\gamma_{j+1}$ are the corners of lozenge $L_{j}$ for $j=1,\ldots,4$. 

We claim $\tilde{\lambda}_3$ is a leaf with the desired shadow. By the observation above, there are unstable half leaves based at $\bar\gamma_1,\bar\gamma_2,$ and $\bar\gamma_3$ and a stable half leaf based at $\bar\gamma_4$  contained in $\Omega_{\tilde{\lambda}_3}$, and $l^s_1,l_2^s, l^s_3, l^u_4\subset \fr\Omega_{\tilde\lambda_3}$. By Lemma \ref{lozengelemma}, we conclude $L_1,L_2,L_3\subset \Omega_{\tilde{\lambda}_3}$. Since $l^u_4\subset \fr\Omega_{\tilde{\lambda}_3}$ and the line of lozenges share unstable sides, $L_4$ is disjoint from $\Omega_{\tilde{\lambda}_3}$. Hence, $\fr\Omega_{\tilde{\lambda}_3}$ contains a stable frontier chain, $C$, consisting of two branching leaves from a cataclysm consisting of exactly three branching leaves. By Corollary \ref{orbit2calegari}, $\End(C)$ is a diagonal leaf in $\Lambda_u^+$ and therefore $\overline{\partial\mathcal{O}^s}\subsetneq \Lambda^+_u$.
\end{proof}

\begin{rem}
    The construction in Theorem \ref{examplethm} can be extended to a family of atoroidal three-manifolds, $M_n$, each with a pseudo-Anosov flow, $\psi_n$, and transverse foliations $\F_n$. Such an $M_n, \psi_n$, and $\F_n$ can be constructed using $n+1$ round handles to build a block $B_n$ and following the proof of Theorem \ref{examplethm} analogously. Hence, $\psi_n$ and $\F_n$ have the property that the orbit space of $\psi_n$ contains a line of $n$ lozenges and that there exists a leaf $\lambda\in \tilde{\F}_n$ whose shadow contains exactly $k$ lozenges, for any $n\geq 4$ and $k< n$.
\end{rem}
\bibliographystyle{alpha}
\bibliography{simlamcitations}
\end{document}